\documentclass[11pt]{amsart}
\usepackage{amsmath,amssymb,amsthm}
\usepackage{graphicx}
\usepackage{indentfirst}
\usepackage{extarrows}
\usepackage[colorlinks]{hyperref}  
\usepackage{bbm}
 \usepackage{color}
\usepackage{mathrsfs}

\newtheorem{theorem}{\sc Theorem}[section]
\newtheorem{lemma}[theorem]{\sc Lemma}

\newtheorem{proposition}[theorem]{\sc Proposition}

\theoremstyle{remark}
\newtheorem{remark}[theorem]{\sc Remark}
\numberwithin{equation}{section}

\newcommand{\be}{\begin{equation}}
\newcommand{\ee}{\end{equation}}
\newcommand{\nn}{\nonumber}
\providecommand{\abs}[1]{\lvert#1\rvert}

\providecommand{\absb}[1]{\bigl\lvert#1\bigr\rvert}

\providecommand{\norm}[1]{\Vert#1\Vert}

\newcommand{\fl}[1]{\lfloor{#1}\rfloor}

\newcommand{\RNum}[1]{\uppercase\expandafter{\romannumeral #1\relax}}

\def\E{\bE}
\def\P{\bP} 
\def\fluc{\overline{h}} 
\def\cov0{V_0} 

\def\cA{\mathcal{A}}\def\cB{\mathcal{B}}
\def\cF{\mathcal{F}}

\def\cH{\mathcal{H}}
\def\cM{\mathcal{M}}
\def\cN{\mathcal{N}}

\def\cS{\mathcal{S}}
\def\cI{\mathcal{I}}

\def\bE{\mathbb{E}}
\def\bN{\mathbb{N}}
\def\bP{\mathbb{P}}

\def\bR{\mathbb{R}}
\def\bZ{\mathbb{Z}}

\def\e{\varepsilon}
\def\ind{\mathbf{1}}
\def\ddd{\displaystyle} 

\def\h{h}

\def\mE{\mathbf{E}}   

\def\mP{\mathbf{P}}
\def\m1{\mathbf{1}}
\def\mVar{\text{\bf Var}}
\def\mCov{\text{\bf Cov}}

\def\R{\bR} \def\N{\bN} \def\Z{\bZ}   
 
\def\OFP{(\Omega,\cF, P)}

\DeclareMathOperator{\Var}{Var}   \DeclareMathOperator{\Cov}{Cov}  
 \def\Vvv{{\rm\mathbb{V}ar}}   
 
   \def\wt{\widetilde}
 
 \def\sss{g}
 \def\pmu{v_1}

\allowdisplaybreaks[2]

\begin{document}
\title[Hammersley's  harness process]{Hammersley's  harness process: invariant distributions and height fluctuations} 

\author{Timo Sepp\"al\"ainen}
\address{Timo Sepp\"al\"ainen\\ University of Wisconsin-Madison\\ 
Department of Mathematics\\ 480 Lincoln Dr.\\  
Madison WI 53706\\ USA.}
\email{seppalai@math.wisc.edu}

\author{Yun Zhai}
\address{Yun Zhai\\ University of Wisconsin-Madison\\ 
Department of Statistics\\ 1300 University Ave\\  
Madison WI 53706\\ USA.}
\email{yunzhai41@gmail.com}
\keywords{Harness, Gaussian process, Edwards-Wilkinson universality class, random walk, fractional Brownian motion, fluctuations, interface, process tightness, strong mixing coefficients, linear process, harmonic crystal, stochastic heat equation}
\date{\today}
\begin{abstract}  We study the invariant distributions of Hammersley's serial harness process in all dimensions and height fluctuations in one dimension.   Subject to mild moment assumptions there is essentially one  unique invariant distribution,  and  all other  invariant distributions  are obtained    by adding harmonic functions of the  averaging kernel.    We identify one Gaussian case where the invariant distribution is i.i.d.  Height fluctuations in one dimension obey the  stochastic heat equation  with additive noise   (Edwards-Wilkinson universality).   We prove this for correlated initial data subject to polynomial decay of  strong mixing coefficients, including process-level tightness in the  Skorohod space of space-time trajectories.  
\end{abstract}
\maketitle

\section{Introduction}

Thinking about  the  crystalline structure of metals around 1956  led J.~M.~Hammersley  to formulate  the {\it serial harness}.  This process    $h_t(x)$ evolves on the  lattice $\Z^d$  via the equation 
\be\label{hamm1}   h_{t+1}(x)=\sum_y p(y-x) h_t(y)  + \xi_{t+1}(x) \ee
where $p$ is a symmetric random walk kernel on $\Z^d$ and $\{\xi_t(x)\}$ are i.i.d.\ random variables with mean zero and finite variance.  These ideas  were  recorded later in article  \cite{hamm-67} that contains, among other things,    calculations that relate the fluctuations of the process to the behavior of the random walk kernel.

   Toom \cite{toom-97}   studied   the convergence of $h_t(x)$ as $t\to\infty$, as a function of the   tail of the  noise $\xi_t(x)$.   \cite{toom-97} also contains several  references to the physics literature.

 A natural  continuous-time version of the process  applies the   local update \eqref{hamm1}   at the epochs of Poisson processes attached to lattice points $x$.  The ergodic properties of the continuous-time process were investigated by   Hsiao \cite{hsiao-82,hsiao-85},   first for the Gaussian case and then more generally.  In particular,   \cite{hsiao-85}  recorded the order of magnitude of the fluctuations of $h_t(x)$ from the  identically $0$ initial condition:  $t^{1/4}$ in $d=1$,  $\sqrt{\log t}$ in $d=2$, and bounded in $d\ge 3$.    In $d\ge 3$,  \cite{hsiao-85}  showed  (i) uniqueness  of an invariant distribution that is invariant under spatial shifts and has given finite mean and variance, and (ii)    convergence  to such equilibrium from initial distributions that are invariant and ergodic under spatial shifts and possess a finite second moment.  
   
 In the same    continuous-time setting, Ferrari and Niederhauser  \cite{ferr-nied-06}  introduced 
  a dual representation of the harness process in terms of backward random walks.  Through  martingale techniques and the random walk representation, \cite{ferr-nied-06} proved convergence of the process from flat initial profiles both on $\Z^d$ and on subsets of $\Z^d$ and obtained some rates of convergence.    In $d\in\{1,2\}$, instead of $h_t$ itself, \cite{ferr-nied-06}  considered   $h_t-h_t(0)$  (the process as seen from the height at the origin) or the pinned process with boundary condition  $h_t(0)=0$.  
In particular,  \cite{ferr-nied-06} identified the invariant distributions of the harness with Gaussian noise in  $d\ge 3$   as  Gaussian Gibbs fields (harmonic crystals)  studied by \cite{capu-deus-00}.

Our paper  works with the discrete time process \eqref{hamm1}, with a general finitely supported random walk kernel $p$.  
We  establish that  in one space dimension the  height fluctuations of the process  obey Edwards-Wilkinson universality.  This means that on the space and time scale $n$,   fluctuations of the height are of order  $n^{1/4}$,   spatial correlations occur on the scale $n^{1/2}$, and limit distributions are Gaussian.  In particular, height fluctuations of the stationary process converge  to fractional Brownian motion with Hurst parameter $1/4$.  We do not address limits in  higher dimensions.    

 We prove the height fluctuations  for spatially stationary,  correlated initial data, subject to polynomial decay of  strong mixing coefficients. 
 Our proofs use the discrete-time counterpart of the dual representation of Ferrari and Niederhauser  \cite{ferr-nied-06} and  employ a CLT for linear processes due to  Peligrad and Utev  \cite{peli-utev-97}.   We obtain also  process-level tightness in the  Skorohod space of space-time trajectories.  

  As preparation for the fluctuation results   we study the invariant distributions of the process.  In  this part  there is no advantage in restricting to one dimension so we do it in general.  
  The difference with past work is that we consider the increment process, as is necessary for  invariant distributions in $d\in\{1,2\}$.  In any case, the increment process is the natural object  to study because total increment is conserved. (In one dimension, if height $h(x)$ jumps,  the changes to the  left and right increments   $h(x)-h(x-1)$ and $h(x+1)-h(x)$ cancel each other.)  For a given kernel $p$ and noise distribution $\xi$, there is essentially one invariant distribution and its transformations by  adding harmonic functions.  We prove this uniqueness under a condition on the growth of the first moment of the increments  on the lattice ${\Z^d}$.    
 
 \smallskip
 
To summarize, these are the contributions  of this paper.   

(i) For invariant distributions we cover low dimensions $d=\{1,2\}$. Our convergence and uniqueness results require weaker moment assumptions than earlier results and we  do not assume spatial shift invariance for uniqueness.  

 (ii)  We prove that    the harness process obeys EW universality.   
For EW class height fluctuations  in general,  past work is for product form initial distributions,  while  our paper covers some correlated initial data.    In particular, we see that fractional Brownian motion  with Hurst parameter $1/4$ arises from a stationary process also when the invariant distribution is not of product form.   We also obtain process-level tightness in space and time, which was done earlier only for independent walks.  

 \smallskip 
 
 Past proofs of  one-dimensional EW class fluctuations for an interface or a particle current covered independent particles \cite{kuma-08, sepp-rw, sepp-10-ens},  independent particles  in a static random environment \cite{pete-sepp-10}  and   in a dynamic random environment \cite{jose-rass-sepp}, the random average process \cite{bala-rass-sepp, ferr-font-rap},  
  and a recent  continuum example from Howitt-Warren flows \cite{yu-flow}.  
   In the context of colliding particles where the position of a tagged particle is a function of  current, related  work goes back to Harris \cite{harr-65} and later in  \cite{durr-gold-lebo}.  Fluctuations of order $t^{1/4}$ and  fractional Brownian motion  limits have also been identified in the simple symmetric exclusion process since the classic work of Arratia \cite{arra-83}  and later in  \cite{dema-ferr-02, jara-land-06, peli-seth-08, rost-vare-85}.  
 
The  EW class  should be contrasted with the KPZ (Kardar-Parisi-Zhang)  class where the exponents are $1/3$ for height or current fluctuations and $2/3$ for spatial correlations.  Recent  reviews of KPZ appear in   \cite{boro-gori, corw-rev}.  Roughly speaking,  EW fluctuations are expected when the macroscopic  velocity of the interface is a linear function of  the slope (as in Theorem \ref{thm:hydro}  below), 
while KPZ fluctuations are expected when  this connection is nonlinear.  
 
 \subsection*{Organization of the paper}
Section \ref{sec:hamm} describes the model and the results on the invariant distributions,  and Section \ref{sec:fluct}  contains the height fluctuations in $d=1$.  
The remainder of the paper covers the proofs:  Section \ref{sec:inv-pf}  for invariant distributions,  Section \ref{sec:fd} for finite-dimensional convergence of height fluctuations, and Section \ref{sec:tight}   for process-level tightness  of height fluctuations.  

\subsection*{Notation and conventions}   We collect here some items for quick reference.  
$\Z_+=\{0,1,2,3,\dotsc\}$,  $\N=\{1,2,3,\dotsc\}$,   $[n]=\{1,2,\dotsc, n\}$.   The $d$-dimensional integer lattice is $\Z^d=\{x=(x_1,\dotsc, x_d):  \text{each $x_i\in\Z$}  \}$.   $\abs{x}$ is the absolute value of $x\in\R$  or the  Euclidean norm of $x\in\R^d$.  The imaginary unit is $\iota=\sqrt{-1}$.   $C$ is a finite positive constant whose value may change from line to line.  

 $X\sim\mu$ means that random variable $X$ has distribution $\mu$:  $P(X\in A)=\mu(A)$ for Borel sets $A$ on the state space of $X$.  $\cM_1$ is the space of Borel probability measures on $(\R^d)^{\Z^d}$.  $\cI$ is  the subspace of measures that are  invariant for the increment process defined below by equation \eqref{eta-11}.   Transition probability $p^k(x,y)$ can be written $p^k_{x,y}$ to save space.

 $E$ represents  generic expectation,  and  $E^\mu$ and $\Var^\mu$  denote expectation and variance  under probability measure $\mu$.      
There are three particular expectation operators. (i)  $\E$ is expectation under the probability measure $\P$ of the i.i.d.\ variables  $\xi=\{\xi_t(x)\}_{t\in\Z,\,x\in\Z^d}$ that drive the dynamics.  (ii)   $\mP$ is the probability measure and $\mE$ the expectation  of the harness process.  The superscript on $\mE^\eta$ and $\mE^\nu$  identifies the initial increment configuration $\eta$ or the initial distribution $\nu$.  
  $\P$ is the $\xi$-marginal of $\mP$. (iii)  In the proofs of the height fluctuations,  $P$ with expectation  $E$ refers to the random walks $X^i_t$ of the dual representation of the harness process.


\section{The harness process and its invariant distributions}\label{sec:hamm}
\subsection{The model}  Fix a dimension $d\in\N$.  The state of the harness process is a real-valued {\it height function}  $\h:\Z^d\to\R$ that evolves randomly  in discrete time. 
The evolution is determined by a fixed weight   vector $\{p(x)\}_{x\in\Z^d}$ and a collection of i.i.d.\ ``noise'' variables $\{\xi_t(x)\}_{t\in\Z, x\in\Z^d}$.  The weight vector $\{p(x)\}$ is a finitely supported  nondegenerate  probability vector on $\Z^d$ with these properties: 
\be\label{w-ass} \begin{aligned}
&0\le p(x)<1, \ \  \sum_{x\in\Z^d}p(x)=1, \ \   \exists M<\infty \text{  such that  $p(x)=0$ for $\abs{x}>M$, and }\\
&\text{$\forall u\in\Z^d$,  the smallest additive  subgroup of  $\Z^d$  that contains the translated}\\
&\text{support $\{u+x: p(x)>0\}$ is  $\Z^d$  itself.}    
\end{aligned}\ee
The last property is {\it strong aperiodicity} in Spitzer's terminology \cite[p.~42]{spitzer}.     In $d=1$ it is the same as requiring that $p$ have {\it span 1} \cite[Sect.~3.5]{durr}.

It is useful to think of $p(x)$ as a random walk transition probability $p(x,y)=p(y-x)$.  Multistep transitions are  denoted by   $p^0(x,y)=\ind_{x=y}$, $p^1(x,y)=p(x,y)$, and for $k\ge 2$, 
\[  p^k(x,y)=\sum_{x_1, x_2, \dotsc, x_{k-1}\in\Z^d} p(x,x_1)p(x_1,x_2)\dotsm p(x_{k-1},y) .\] 
Let $\pmu=\sum_{x\in\Z^d}  x\,p(x)$ denote the mean (vector) of $p$. In 
  $d=1$  the   variance is 
\be\label{p-6}  
\sigma_1^2=\sum_{x\in\Z}  (x-\pmu)^2p(x), \quad\text{assumed $>0$.}  \ee

The state of the height process at time $t\in\Z$  is denoted by $\h_t=\{\h_t(x)\}_{x\in\Z^d}\in\R^{\Z^d}$.   Given $h_t$,    the evolution from time $t$ to $t+1$ is governed by the equation 
 \be
\h_{t+1}(x)=\sum_{y\in\Z^d}p(x,y) \h_t(y)+\xi_{t+1}(x),\quad x\in\bZ^d.   \label{i1}\ee
Since adding a constant to $\xi_t(x)$ would simply add a constant speed to   $h_t$, we   assume   that $\E[\xi_t(x)]=0$.    Our results also require (at least) square-integability.      We summarize these assumptions as 
 \be\label{xi-ass} \begin{aligned}
&\text{$\xi=\{\xi_t(x)\}_{t\in\bZ, \,x\in\bZ^d}$ are  i.i.d.\ real-valued mean-zero  random variables}\\ &\text{with $\sigma_\xi^2= \E[\,\abs{\xi_t(x)}^2\,]<\infty$. }    
\end{aligned}\ee
The probability measure on the variables $\xi$ is denoted by $\P$ and expectation by $\E$, while for the harness process we write $\mP$ and $\mE$.   $\P$ is the  $\xi$-marginal of $\mP$.  
 
That the harness process should obey EW universality is indicated by a hydrodynamic limit  described by a linear partial differential equation $\frac{\partial}{\partial t}u - \pmu\cdot\nabla u=0$.  

\begin{theorem}[Hydrodynamic limit]   \label{thm:hydro} 
Let $u_0$ be a continuous  function on $\R^d$ and define $u(t,x)=u_0(x+t\pmu)$.  
Let $\{h^n_\centerdot\}$ be a sequence of harness processes whose initial height functions approximate $u_0$ locally uniformly in probability: 
\[ 
\lim_{n\to\infty} \mP\Bigl\{ \, \sup_{x\in\R^d:\,\abs x\le R} \absb{ n^{-1} h^n_0(\fl{nx}) -u_0(x)} \ge\e\Bigr\}  =0 \quad \forall \,\e>0, \; R<\infty.  
\] 
Then 
\[  \lim_{n\to\infty} \mP\bigl\{ \,    \absb{ n^{-1} h^n_{\fl{nt}}(\fl{nx}) -u(t,x)} \ge\e\Bigr\}  =0 \quad \forall \,\e>0,\, (t,x)\in\R_+\times\R^d. \]  
\end{theorem} 
   We omit the straightforward proof of the theorem (see \cite{zhai-phd}).

\subsection{Invariant distributions}   We focus on the invariant distributions of the increment process because  those exist in all dimensions.  Once formulated, the increment process stands on its own and  is more general because not every process that  obeys the increment evolution comes from a height process. We study the invariant distributions of this more general ``increment process''.  At the end of the section we  make brief  contact with the height process.  
  

Given a height process $h_\centerdot$ that obeys  \eqref{i1},   define 
increment variables by   
\be\label{eta-5} \eta_t(x,y)=\h_t(y)-\h_t(x)\quad\text{ for $x,y\in\Z^d$.} \ee
 Increment variables satisfy  the   additivity relation 
 \be\label{add} \eta_t(x,y)+\eta_t(y,z)=\eta_t(x,z) \quad\text{ for $x,y, z\in\Z^d$} \ee
and by   \eqref{i1}  the  Markovian evolution   
 \be\label{De11} \begin{aligned} 
&\eta_{t}(x,y)=\sum_{z\in\Z^d}  p(0,z)\eta_{t-1}(x+z,y+z) \, +\, \xi_{t}(y)-\xi_{t}(x).
\end{aligned} \ee
Conversely,  if   variables $\eta_t(x,y)$ satisfy additivity \eqref{add} at $t=0$ and \eqref{De11} for $t\in\N$, then \eqref{eta-5} holds for  the height process $h_\centerdot$  defined by $h_0(0)=0$, $h_0(x)=\eta_0(0,x)$, and \eqref{i1}.   

Iterate \eqref{De11}  backward in time: 
 \be\label{De11a} \begin{aligned} 
&\eta_{t}(x,y)=\sum_{z\in\Z^d}  p(0,z)\eta_{t-1}(x+z,y+z) \, +\, \xi_{t}(y)-\xi_{t}(x)\\
 &=\dotsm= 
 \sum_{z\in\Z^d}  p^{s}(0,z)\eta_{t-s}(x+z,y+z) \, +\, \sum_{k=0}^{s-1}\sum_{z\in\Z^d} \xi_{t-k}(z) \bigl(p^k(y,z)-p^k(x,z)\bigr) 
\end{aligned} \ee
for  $x,y\in\Z^d$ and $s\in\N$.  
Imagine taking $s\to\infty$   with  zero initial condition ``$\eta_{-\infty}=0$''  in the infinite past.  This suggests the  definition  
 \be\label{De7}  \Delta_t(x,y)=\sum_{k=0}^\infty \sum_{z\in\Z^d} \xi_{t-k}(z) \bigl( p^k(y,z)-p^k(x,z)\bigr),\quad x,y\in\Z^d, \; t\in\Z.
\ee
The series of independent mean zero variables on the right of \eqref{De7} converges almost surely and in $L^2(\P)$ because the sum of variances converges (Theorem 2.5.3 in \cite{durr}):
\be\label{De7.5}\begin{aligned}
\E[ \Delta_t(x,y)^2]&= \sigma_\xi^2\sum_{k=0}^\infty \sum_{z\in\Z^d}   \bigl( p^k(y,z)-p^k(x,z)\bigr)^2\\
&= 2\sigma_\xi^2\sum_{k=0}^\infty  \bigl( q^k(0,0)-q^k(x-y,0)\bigr) =2\sigma_\xi^2a(x-y)
\end{aligned}\ee 
where we introduced the  symmetric random walk kernel 
\be q(x,y)=q(0,y-x)=\sum_{z\in\Z^d}p(0,z)p(x,y+z),\quad x,y\in\bZ^d,  \label{21-d2}\ee
and its  potential kernel   
\be
a(x)=\sum_{k=0}^\infty [q^k(0,0)-q^k(x,0)], \qquad x\in\bZ^d.\label{t37-2-1}\ee
The strong aperiodicity assumption  in  \eqref{w-ass}  guarantees that $\Z^d$ is the smallest subgroup that contains the support $\{x: q(x)>0\}$.  
Convergence of  \eqref{t37-2-1}  under this property   is classical:  see Section 28  in \cite{spitzer} for $d=1$ and  Section 12 for $d=2$. In $d\ge 3$ convergence follows from the transience of the $q$-walk. 

\begin{remark}
The convergence issue here is real.  Suppose $p(-1)+p(1)=1$  in $d=1$ so that  $q$ is supported on $\{-2,0,2\}$ and assumption \eqref{w-ass} is violated.  Then from \eqref{De7} 
\[  \Delta_t(0,1)=\sum_{k=0}^\infty \sum_{z\in\Z} \xi_{t-k}(z)   p^k(1,z)  
-   \sum_{k=0}^\infty \sum_{z\in\Z} \xi_{t-k}(z)  p^k(0,z)  \]  
splits into two independent sums because for a given $(k,z)$ pair at most one of 
$p^k(1,z)$ and $p^k(0,z)$ is nonzero.   That this series cannot converge even in distribution in $d=1$  can be shown with the characteristic function argument of the proof of Theorem \ref{h-thm}(a) in Section \ref{sec:inv-pf}. 
\end{remark}

One checks from the definition \eqref{De7} that the variables $\Delta_t(x,y)$ satisfy additivity \eqref{add} and the evolution equation
\be\label{De11b} 
\Delta_{t+1}(x,y)=\sum_z p(0,z)\Delta_t(x+z,y+z) \, +\, \xi_{t+1}(y)-\xi_{t+1}(x), \quad x,y\in\Z^d,\; t\in\Z. 
\ee
Note that $\Delta_t(x,y)$  is a function of $\{\xi_s\}_{s:s\le t}$ and thereby independent of $\xi_{t+1}$.    Hence  \eqref{De11b}  describes a  Markovian evolution.     $\{\Delta_t(x,y)\}$ is  an  increment process of a height process. 

By additivity  \eqref{add}  it is enough to keep track of nearest-neighbor increments.  Let us  simplify   the   process to  $\eta_t=\{\eta_t(x)\}_{x\in\Z^d}$ with the vector-valued spin variable  $\eta_t(x)=(\eta_t(x-e_i,x))_{i\in[d]}\in\R^d$.  (The notation is $[d]=\{1,2,\dotsc,d\}$.) The state space of the time-$t$  configuration  $\eta_t$   is $(\R^d)^{\Z^d}$.   
For $y-x=e_i$ the evolution equation 
 \eqref{De11} specializes to  
\be\label{eta-11} 
\eta_{t+1}(x-e_i,x)=\sum_z p(0,z)\eta_t(x+z-e_i,x+z) \, +\, \xi_{t+1}(x)-\xi_{t+1}(x-e_i), 
\ee
for  $x\in\Z^d$ and $ i=1,\dotsc,d$.  We now consider  the $(\R^d)^{\Z^d}$-valued  process $\eta_t$  defined by \eqref{eta-11}, without  imposing the additivity requirement \eqref{add}.  





Let us establish some standard notation and  terminology.  
  A generic element of the state space  $(\R^d)^{\Z^d}$ of   process \eqref{eta-11}  is denoted by $\eta=(\eta(x))_{x\in\Z^d}$ with $\eta(x)=(\eta(x-e_i,x))_{i\in[d]}\in\R^d$.   $\cM_1=\cM_1((\R^d)^{\Z^d})$ is the space of probability measures on $(\R^d)^{\Z^d}$. 
   A   measure $\mu\in \cM_1$ is invariant for  process \eqref{eta-11}  if equation \eqref{eta-11} preserves this distribution:  that is, if $\eta_0\sim\mu$ and noise $\xi_1=\{\xi_1(x)\}_{x\in\Z^d}$ is independent of $\eta_0$, then $\eta_1$ defined by  \eqref{eta-11} satisfies again  $\eta_1\sim\mu$.   $\cI$ denotes the convex set of probability measures in $\cM_1$  that are invariant for  process \eqref{eta-11}.  

Write $\{\theta_{x,t}\}_{x\in\Z^d, \, t\in\Z}$ for 
  shift mappings  in space ($x$-index)  and time ($t$-index).   When only space or only time is shifted, the other index is set to zero. 
   For example,   for   $\eta\in(\R^d)^{\Z^d}$,  $(\theta_{a,0}\eta)(x-e_i,x)=\eta(x+a-e_i,x+a)$.    
  A probability measure $\mu\in\cM_1$ is (spatially) shift-invariant   if $\mu(\theta_{x,0}A)=\mu(A)$ for all Borel sets $A\subseteq (\R^d)^{\Z^d}$ and $x\in\Z^d$. A Borel set $B\subseteq (\R^d)^{\Z^d}$ is  shift-invariant  
    if  $\theta_{x,0}B=B$ $\forall x\in\Z^d$.  A shift-invariant probability measure 
 $ \mu\in\cM_1$ is ergodic if $\mu(B)\in\{0,1\}$ for every  invariant Borel set $B$.   A shift-invariant probability measure $\mu$ on $(\R^d)^{\Z^d\times\Z}$ is ergodic under the individual space-time shift $\theta_{z,s}$ if  $\mu(A)\in\{0,1\}$ for every Borel set $A$ that satisfies $\theta_{z,s}A=A$.  
 
 The Markov chain  $\eta_\centerdot$ defined by \eqref{eta-11} without the additivity requirement \eqref{add}  gives us a   larger class of processes than the increment processes coming from  height processes.  But we find that, under very lenient moment assumptions, all  stationary $\eta_\centerdot$ processes are obtained by adding harmonic functions to  $\Delta_\centerdot$ of  \eqref{De7}.   
  
 The first theorem summarizes what was   developed above.   Set 
 \be\label{De16}  \Delta_t(x)=(\Delta_t(x-e_i,x))_{i\in[d]}\in\R^d\quad\text{ and }\quad 
 \Delta_t=\{\Delta_t(x)\}_{x\in\Z^d}\in(\R^d)^{\Z^d}. \ee     
 The   
 invariance and   ergodicity claim   below  follows from  
  $\Delta_t(x)(\xi)=\Delta_0(0)(\theta_{x,t}\xi)$.     Let $\pi_0\in\cI$ 
denote  the distribution of $\Delta_t$.   
 
 \begin{theorem}\label{eta-thm-1}   Assume \eqref{w-ass} and \eqref{xi-ass}. 
  Then the  series in \eqref{De7} converges  for almost every $\xi$ and   in $L^2(\bP)$.      For almost every  $\xi$   the variables $\Delta_t(x,y)$ satisfy additivity \eqref{add} and the  evolution  equation 
\eqref{De11b}.  
 The process $\Delta_\centerdot=\{\Delta_t(x)\}_{t\in\Z, x\in\Z^d }$  defined in \eqref{De16} is a stationary version of the Markov chain  \eqref{eta-11} and is invariant and ergodic under each individual space-time shift mapping $\theta_{x,t}$ for $(x,t)\ne(0,0)$. 

  \end{theorem}


 Next we add  harmonic functions to  the variables $\Delta_t(x-e_i,x)$ to produce invariant distributions with nonzero means  and address  
    uniqueness and convergence.    
    
    A function $v:\Z^d\to\R$ is {\it harmonic}  for transition probability kernel  $p$  if 
  \be\label{p-harm}  \sum_{y\in\Z^d}  p(x,y)v(y)=v(x) \qquad\text{ for all $x\in\Z^d$. }\ee
    Constants are the only bounded harmonic functions (Lemma \ref{lm-p-harm} in the appendix)   but there can be many  unbounded ones. 
    Let $\cH_d$ denote the space of harmonic functions $u:\Z^d\to\R^d$, in other words, the space of vectors  
 $u=(u_1,\dotsc, u_d)$ of $d$ real-valued harmonic functions $u_i:\Z^d\to \R$.
 Given $u\in\cH_d$,    define a process $\Delta^u_t=u+\Delta_t$ by
\be\label{De-u}  \Delta^u_t(x-e_i,x)=u_i(x)+ \Delta_t(x-e_i,x), \qquad   t\in\Z, \; x\in\Z^d, \; i\in[d],   \ee 
where   $\Delta_t(x-e_i,x)$ is  defined by  \eqref{De7}. 
   $\Delta^u_\centerdot$ is another time-stationary and ergodic  process whose evolution obeys   \eqref{eta-11}.
  Let   $\pi_u\in\cI$  denote the distribution of the configuration   $\Delta^u_t$.  
 
For nonconstant $u$,   $\Delta^u_\centerdot$ may fail additivity   \eqref{add}.    We are now considering the broader class of processes defined by \eqref{eta-11}, regardless of whether this process comes from a height process. 
 
  We record the invariant distributions $\pi_u$ in the next  theorem. That $\pi_u$ is an {\it extreme point}  of $\cI$ means that if $\pi_u=b\mu+(1-b)\nu$ 
  for $0<b<1$ and $\mu,\nu\in\cI$, then $\mu=\nu=\pi_u$.  
 
 \begin{theorem} \label{thm-pi-u} {\rm(Existence.)}     Assume \eqref{w-ass} and \eqref{xi-ass}. 
 Then  the measures $\{\pi_u: u\in\cH_d\}$ are extreme points of $\cI$.  They are also the unique invariant distributions of minimal variance:    for any $\nu\in\cI$ with finite variances, 
     \be\label{nu-var1-4} \Var^\nu[\eta(x-e_i,x)]\ge \Var^{\pi_0}[\eta(x-e_i,x)],  \ee
  and equality   holds for all  $i\in[d]$ and $x\in\Z^d$ iff $\nu\in\{\pi_u: u\in\cH_d\}$. 
\end{theorem}  
 

 
 The next   theorem  shows that these are the unique extreme  invariant distributions    under  a growth condition on the centered  first moment of $\eta(x)$.

\begin{theorem}\label{thm-uniq} {\rm(Uniqueness.)}   Assume \eqref{w-ass} and \eqref{xi-ass}. 
Let   $\nu\in\cI$ satisfy  these properties:   $E^\nu\abs{\eta(x)}<\infty$ $\forall x\in\Z^d$ and  there exists $u\in\cH_d$ such that 
\be\label{nu-mom}    \lim_{r\to\infty}  r^{-1/2} \cdot \max_{x\in\Z^d:\,\abs{x}\le r} E^\nu\bigl\{ \,\absb{\,\eta(x)- u(x)\,} \,\bigr\} \ = \ 0.  
\ee
Then \eqref{nu-mom} holds also if $u(x)$ is replaced by  $\bar u(x)=E^\nu[\eta(x)]$. There exists a  probability measure $\gamma$ on $\R^d$ such that $\nu=\int_{\R^d}  \pi_{\bar u+\alpha}\,\gamma(d\alpha)$.  
\end{theorem}  

Above $\pi_{\bar u+\alpha}$ is the measure $\pi_u$ for the harmonic function $u(x)=\bar u(x)+\alpha$ in $\cH_d$.  


  
 \begin{remark}
The key to the proof of Theorem \ref{thm-uniq} is the local central limit theorem applied to the random walk $p^t(0,x)$.   The factor $r^{-1/2} $ in assumption \eqref{nu-mom}  comes from a smoothness bound of Gamkrelidze \cite[Theorem 4]{Gamk-85}:
  \be\label{gram3}
\sum_{x\in\Z^d}\sum_{\ell=1}^d \abs{p^t(0,x)-p^t(0, x-e_\ell)} \le C t^{-1/2}  
\quad \forall t\in\Z_+.  
\ee
For an irreducible, aperiodic kernel $p$ the same estimate is given in Prop.~2.4.2 of \cite{lawl-limi}. 
 \end{remark} 
  
  The third item is convergence to an invariant distribution.  We currently have a result only when the centered  initial distribution  is of this type:
\be\label{zeta-7} \begin{aligned}
 &\text{$\zeta=\{\zeta(x)\}_{x\in\Z^d}$ is an $(\R^d)^{\Z^d}$-valued random configuration whose distribution }\\
 &\text{is invariant and ergodic 
 under the spatial shift group $\{\theta_{x,0}\}_{x\in\Z^d}$, and } \\&\text{the $\R^d$-valued variable   $\zeta(x)=(\zeta(x,i))_{i\in[d]}$ has mean $E\zeta(x)=0$.}  
 \end{aligned}\ee

 \begin{theorem}\label{eta-thm-2} {\rm(Convergence.)}   Assume \eqref{w-ass} and \eqref{xi-ass}. 
Let $\{\eta_t\}_{t\in\Z_+}$ be an  $(\R^d)^{\Z^d}$-valued process  
whose  evolution is  defined by \eqref{eta-11}.  
Assume that the initial distribution of the process $\eta_t$  is of the following form: $\eta_0(x-e_i,x)=u_i(x) + \zeta(x,i)$ for $x\in\Z^d$ and $i=1,\dotsc,d$  where $u=(u_1,\dotsc, u_d)\in\cH_d$ and 
$\zeta$ is as in \eqref{zeta-7}.   Let $\mu_t$ denote the distribution of $\eta_t$.  Then, as $t\to\infty$,  $\mu_t\to \pi_u$ weakly in $\cM_1$ .   

\end{theorem}


\begin{remark}[Spatially invariant case] 
Let us spell out the most natural special case.   Given a constant    $\alpha\in\R^d$,  there is a unique   $\pi_\alpha\in\cI$   that is invariant  and ergodic for the  process $\eta_t$ defined by  \eqref{eta-11},   invariant and ergodic under the spatial shift group $\{\theta_{x,0}\}_{x\in\Z^d}$,  and has mean  
$E^{\pi_\alpha}[\eta(x)]=\alpha$.   $\pi_\alpha$  is   the distribution of the $(\R^d)^{\Z^d}$-valued random configuration $\{\zeta(x)\}_{x\in\Z^d}$ defined by 
$\zeta(x)=\alpha +(\Delta(x-e_i,x))_{i\in[d]}$.    Furthermore,  the process \eqref{eta-11} started with an arbitrary  mean-$\alpha$ spatially invariant and ergodic initial distribution converges weakly to $\pi_\alpha$.   
\end{remark}

 In the one-dimensional case that is the subject of the next section,  $\pi_0$ is the distribution of the sequence 
 $\{\Delta_0(x)\}_{x\in\Z}$ where  $\Delta_0(x)=\Delta_0(x-1,x)$ is defined by the series \eqref{De7}.  
The next theorem collects  properties of the shift-invariant covariance 
\[  V_0(x,y)=\E[\Delta_0(x)\Delta_0(y)]=E^{\pi_0}[\eta(x)\eta(y)]=E^{\pi_0}[\eta(0)\eta(y-x)]=V_0(0,y-x) \]
of $\pi_0$ in $d=1$.   The span 1 assumption on $p$ implies that the   support of $q$ cannot lie in a proper subgroup of  $\Z$, and thereby the characteristic function $\phi_q(\theta)=\sum_{x\in\bZ}q(0,x)e^{\iota \theta x}$ equals 1 only at  $\theta\in 2\pi\Z$.

\begin{theorem}\label{thm-cov}   Assume \eqref{w-ass}, \eqref{xi-ass}, and   $d=1$.   Then  $\exists$ constants $0<A,c<\infty$ such that   $\abs{V_0(0,x)} \le Ae^{-c\abs x}$.  
We have the identities
\begin{align}
V_0(0,x)&=\sigma_\xi^2[a(x-1)+a(x+1)-2a(x)]   
=\frac{\sigma_\xi^2}{\pi}\int_{-\pi}^{\pi}\frac{1-\cos\theta}{1-\phi_q(\theta)}e^{\iota \theta x} d\theta, \quad x\in\bZ,     \label{cov=}
\end{align} 
and 
 \be\label{cov-sum}  \sum_{x\in\bZ}  V_0(0,x)=\frac{\sigma_\xi^2}{\sigma_1^2}\ee
 where the series above is absolutely convergent.  
\end{theorem}

The first equality in  \eqref{cov=} comes from  a series like \eqref{De7.5} and the second equality is from  p.~355 of Spitzer \cite{spitzer}.    The rest is proved in Section \ref{sec:inv-pf} from properties of the kernel $a(x)$.  

\begin{remark}[$d=1$ Gaussian case]
Presently we cannot say  more about the distribution $\pi_0$, except in the Gaussian case.   If $\{\xi_t(x)\}$ are centered Gaussian variables, then $\pi_0$ is the mean zero Gaussian measure with covariance \eqref{cov=}.  Furthermore, if  $p(0,z)+p(0,z+1)=1$ for some $z\in\Z$, then $q$ is supported on $\{-1,0,1\}$ and  \eqref{cov=} shows that $V_0(0,x)=0$ for $x\ne 0$. In other words, the variables $\{\Delta_0(x)\}$ are uncorrelated, and thereby in the Gaussian case  they are i.i.d.   
\end{remark} 
  
We turn briefly to 
the invariant distributions of the height process $h_t$. The key point is that in $d\in\{1,2\}$ there are none.    By analogy with \eqref{De7}, define 
 \be\label{chi7}  \chi_t(x)=\sum_{k=0}^\infty \sum_{z\in\Z^d} \xi_{t-k}(z)   p^k(x,z) ,\quad x\in\Z^d, \; t\in\Z.
\ee
Computation of the sum of variances on the right gives 
\begin{align*}
\E[ \chi_t(x)^2]= \sigma_\xi^2\sum_{k=0}^\infty \sum_{z\in\Z^d}    p^k(x,z)^2
= \sigma_\xi^2\sum_{k=0}^\infty   q^k(0,0)  =\sigma_\xi^2 G(0,0)
\end{align*} 
where  the Green function 
\be
G(x,y)=\sum_{k=0}^\infty q^k(x,y), \qquad x,y\in\bZ^d, \label{t37-2-1-1}\ee
converges by transience in $d\ge 3$.  

\begin{theorem} \label{h-thm}   Assume \eqref{w-ass} and \eqref{xi-ass}.  Consider the height process $h_\centerdot$ defined by \eqref{i1}. 

{\rm (a)}  In dimensions $d=1$ and $d=2$ this process has no invariant distributions.  

{\rm (b)}  In dimensions $d\ge 3$ 
   the  series in \eqref{chi7} converges  for almost every $\xi$ and   in $L^2(\bP)$.      For almost every  $\xi$   the variables $\chi_t(x)$ satisfy    equation 
\eqref{i1}.  
 The process $\chi_\centerdot=\{\chi_t(x)\}_{t\in\Z, x\in\Z^d }$    is a stationary version of the Markov chain  \eqref{i1} and is invariant and ergodic under each individual space-time shift mapping $\theta_{x,t}$ for $(x,t)\ne(0,0)$. 
\end{theorem}

Part (b) of the theorem is clear.  Part (a) is proved   in Section \ref{sec:inv-pf}. 

The  stationary height and increment processes \eqref{chi7} and \eqref{De7}  are obviously connected  by $\Delta_t(x,y)=\chi_t(y)-\chi_t(x)$.  
In fact  any bi-infinite process $\{\eta_t(x,y)\}_{t\in\Z,\,x,y\in\Z^d}$ that satisfies additivity \eqref{add} and evolution  \eqref{De11} $\forall t\in\Z$ comes from a bi-infinite height process $\{\h_t\}_{t\in\Z}$, uniquely determined   once  a value $h_0(0)$ is chosen:  the forward process $\{\h_t\}_{t\in\Z_+}$ is determined by \eqref{eta-5} and the evolution \eqref{i1}.  We extend the height process backward in time inductively. Assuming that $\{\h_s\}_{s\ge t}$ has been constructed,  extend to time $t-1$ by first defining the value 
\[ 
h_{t-1}(0)=h_t(z)-\sum_y p(z,y) \eta_{t-1}(0,y) -\xi_t(z)  
\]
where  $z\in\Z^d$ is any point,  and then   $ h_{t-1}(x)=h_{t-1}(0)+\eta_{t-1}(0,x)$ for all $x\ne 0$.  
That the definition of $h_{t-1}(0)$ is independent of $z$ is equivalent to  \eqref{De11}.  That $h_t$ now comes  from $h_{t-1}$ by \eqref{i1} is immediate from the definitions.  


 
\section{Height fluctuations in one dimension}\label{sec:fluct}

Restrict to dimension $d=1$. 
Assume that the initial height function $h_0=\{h_0(x)\}_{x\in\Z}$  is normalized by $\h_0(0)=0$  and   that the distribution of the  initial increment configuration $\eta_0=\{\eta_0(x)=\h_0(x)-h_0(x-1)\}_{x\in\bZ}$ is  invariant under  translations of the spatial index $x$.   The mean, variance, and series of covariances of the initial increment variables are  
\be\label{eta-9}   \mu_0=\mE[\eta_0(x)], \quad  \sigma_0^2=\mVar[\eta_0(x)], \quad  \text{and} \quad    \varrho_0=\sum_{x\in\bZ} \mCov[\eta_0(0),\eta_0(x)].  
 \ee
Our  assumptions will  guarantee absolute  convergence of   the series.   Then $\varrho_0\ge 0$  because it is the limit of 
 $n^{-1} \mVar[\eta_0(1)+\dotsm+\eta_0(n)]$.  
 

  Let 
   $b=-\pmu=-\sum_x xp(0,x)$.  
Scale space by $\sqrt n$ and time by $n$  and consider   the scaled and centered  space-time  height process 
\be
 \fluc_n(t,r)=n^{-1/4}\left\{\h_{\lfloor nt\rfloor}\bigl(\lfloor r\sqrt{n}\rfloor +\lfloor ntb\rfloor\bigr)-\mu_0r\sqrt{n}\,\right\}, \quad (t,r)\in\R_+\times\R.\label{2-1}\ee
We prove that this process converges to a Gaussian process.  
 

The limit process has three natural descriptions: as a Gaussian process, as a sum of two  stochastic integrals, and as the solution of the stochastic heat  equation with additive noise.   The proof is  based on the Gaussian process description.  
Denote the centered Gaussian p.d.f and c.d.f with variance $\nu^2$ by
\be \varphi_{\nu^2}(x)=\frac{1}{\sqrt{2\pi\nu^2}}\exp\left(-\frac{x^2}{2\nu^2}\right)\quad\mbox{and}\quad
\Phi_{\nu^2}(x)=\int_{-\infty}^x\varphi_{\nu^2}(y)dy,   \label{gamma1}\ee
and let   
\be \Psi_{\nu^2}(x)=\nu^2\varphi_{\nu^2}(x)-x\left(1-\Phi_{\nu^2}(x)\right), \qquad x\in\R.\label{gamma2}\ee
Define   two positive definite  functions $\Gamma_1$ and $\Gamma_2$ for 
 $(s,q), (t,r)\in\R_+\times\R$ by  
\be \Gamma_1\bigl((s,q),(t,r)\bigr)=\Psi_{\sigma_1^2(t+s)}(r-q)-\Psi_{\sigma_1^2|t-s|}(r-q) \label{gamma3}\ee
and
\be \Gamma_2\bigl((s,q),(t,r)\bigr)=\Psi_{\sigma_1^2s}(-q)+\Psi_{\sigma_1^2t}(r)-\Psi_{\sigma_1^2(t+s)}(r-q).\label{gamma4}\ee 
Let  $\{Z(t,r): (t,r)\in\R_+\times\R\}$ be the mean zero Gaussian process with  covariance 
\be
E\bigl[Z(s,q)Z(t,r)\bigr]=\frac{\sigma_\xi^2}{\sigma_1^2}\Gamma_1\bigl((t,r),(s,q)\bigr)+\varrho_0\Gamma_2\bigl((t,r),(s,q)\bigr).  \label{2-t13b}\ee 
Equivalently,  this process is given by 
 \be\label{Z-int}  Z(t,r)= \frac{\sigma_{\xi}}{\sigma_1}\iint\limits_{[0,t]\times \bR}\varphi_{\sigma_1^2(t-s)}(r-x)dW(s,x)+\sqrt{\varrho_0}\int\limits_\bR\varphi_{\sigma_1^2t}(r-x)B(x)dx,\ee
where $\{W(t,r):t\in\bR_+,r\in\bR\}$ is a two-parameter Brownian motion and $\{B(r): r\in\bR\}$ a  two-sided Brownian motion, and $W$ and $B$ are independent. 
In the fluctuation limit of $\fluc_n$ described  by  \eqref{2-t13b} and \eqref{Z-int},  the first term comes from the fluctuations of   $\xi$ (the dynamics), and the second term comes  from the fluctuations of $\eta_0$ propagated by the dynamics.  

  $Z$ is also  the unique mild  solution \cite{wals-spde} of the   stochastic partial differential equation 
\be \frac{\partial Z}{\partial t}=\tfrac12 {\sigma_1^2} \,\frac{\partial^2Z}{\partial r^2}+\frac{\sigma_{\xi}}{\sigma_1}\,\dot{W} \ \ \text{ on $\bR_+\times \bR$},  \quad  Z(0,r)=\sqrt{\varrho_0}B(r).\ee
  
We consider three different  hypotheses on the initial increments $\{\eta_0(x)\}$: (a) i.i.d., (b) strongly mixing, and (c) the invariant distribution $\pi_0$ of Theorem \ref{eta-thm-1},  defined by 
\be \eta_0(x)=\sum_{y\in\bZ}\sum_{k=0}^{\infty}\xi_{-k}(y)[p^k(x,y)-p^k(x-1,y)],\quad x\in\bZ.\label{(c)}\ee
Throughout the text  these are referred to as cases (a), (b) and (c).  In case (c) parameter $\varrho_0$ disappears from the limit because    $\varrho_0={\sigma_\xi^2}/{\sigma_1^2}$ \eqref{cov-sum}.  

 Given two sub-$\sigma$-algebras $\cA$ and $\cB$ on a probability space $\OFP$,  let 
\be \alpha(\mathcal{A},\mathcal{B})=\sup_{A\in\mathcal{A}, B\in\mathcal{B}}\abs{P(A\cap B)-P(A)P(B)}.\ee
For the initial increment sequence $\{\eta(x)\}$ define $\sigma$-algebras  $\mathcal{F}^\eta_{m,n}=\sigma\{\eta(x):  m\le x\le n\}$, and then the {\it strong mixing coefficients}  
\be \label{eta-a2} \alpha(n)=\sup_k \alpha (\mathcal{F}^\eta_{-\infty, k}\,, \, \mathcal{F}^\eta_{k+n, \infty} ), \qquad n\in\Z_+.\ee
The sequence $\{\eta(x)\}$ is {\it strongly mixing}  if   $\alpha(n)\to 0$ as $n\to \infty$.  
See Bradley \cite{brad-05} for properties of these and other mixing coefficients. 

Here is the  convergence of finite-dimensional distributions.

\begin{theorem}\label{thm2-4}  Assume  $d=1$, \eqref{w-ass},  \eqref{xi-ass},  and 
  $\E[\xi_t(x)^4]<\infty$.   Assume that the initial increment sequence $\{\eta_0(x)\}_{x\in\Z}$ satisfies one of the assumptions {\rm(a), (b)},  or {\rm(c)}:  
  
 \medskip
  
{\rm(a)}    $\{\eta_0(x)\}_{x\in\Z}$ is an  i.i.d.\ sequence and $\mE[\eta_0(x)^2]<\infty$. 
  
 \medskip 

{\rm(b)}    $\{\eta_0(x)\}_{x\in\Z}$ is a strongly mixing stationary sequence  and  $\exists\delta>0$ such that $\mE[\,\abs{\eta_0(0)}^{2+\delta}\,]<\infty$ and  the strong mixing coefficients in \eqref{eta-a2}  satisfy 
\be\label{eta-a7} \sum_{j=0}^\infty (j+1)^{2/\delta}\alpha(j)<\infty. \ee 
   

 {\rm(c)}   $\{\eta_0(x)\}_{x\in\Z}$ has the distribution $\pi_0$ of the sequence 
 in  \eqref{(c)}.  
   
 \medskip  

Then, $\varrho_0=\sum_{x\in\bZ} \mCov[\eta_0(0),\eta_0(x)]\ge 0$ is absolutely convergent. 
As $n\to\infty$ the  finite-dimensional distributions of the process  $\fluc_n$  of \eqref{2-1}  converge weakly   to those of  the mean-zero Gaussian process $Z$ with covariance \eqref{2-t13b}. 
\end{theorem}

The convergence in the theorem  means that   for any fixed  $N\in\N$ and     $(t_1,r_1), (t_2,r_2),\ldots,$ $(t_N,r_N)\in \bR_+\times \bR$,   
the weak convergence of $\R^N$-valued vectors holds:  as $n\to\infty$, 
\be
\left(\fluc_n(t_1,r_1), \fluc_n(t_2,r_2),\ldots,\fluc_n(t_N,r_N)\right)\Rightarrow \left(Z(t_1,r_1), Z(t_2,r_2),\ldots,Z(t_N,r_N)\right).\label{2-t13a}\ee

\begin{remark} Case (a) is stated only for completeness, we do not prove it.    A natural question is whether case (c) is actually  covered by case (b).  The answer is    affirmative   in the Gaussian case  where,  $\forall n\in\N$,   $\alpha(k)=O(\abs{k}^{-n})$ as $\abs{k}\to\infty$.  See \cite{zhai-phd} for details.  
\end{remark}

\begin{remark}[Fractional Brownian motion]  In case (c) of the theorem, or in general whenever    $\varrho_0={\sigma_\xi^2}/{\sigma_1^2}$, the limit of the  time-indexed process $\fluc_n(t,0)$ is the Gaussian process with covariance 
\be E [Z(s,0)Z(t, 0)]=\frac{\sigma_\xi^2}{\sqrt{2\pi \sigma_1^2}}({s}^{1/2}+ {t}^{1/2}-{|t-s|}^{1/2}\,)\ee
 which is a fractional Brownian motion with Hurst parameter $1/4$.
\end{remark}
  
By strengthening the assumptions we   upgrade the weak convergence 
$\fluc_n\Rightarrow Z$ to process level on a compact time-space rectangle  $Q=[0,T]\times [-R,R]$.    $Z$ is a continuous process on $Q$.  The paths of  $\fluc_n$ lie in the   space $D_2$ of 2-parameter cadlag paths, defined precisely as follows.   Given  $(t_0,r_0)\in Q$, define   four quadrants  by 
\begin{align*}
&Q_{(t_0,r_0)}^1=\{(t,r)\in Q: t\ge t_0, r\ge r_0\}, \quad Q_{(t_0,r_0)}^2=\{(t,r)\in Q: t\ge t_0, r< r_0\},\\
&Q_{(t_0,r_0)}^3=\{(t,r)\in Q: t< t_0, r< r_0\}, \quad  Q_{(t_0,r_0)}^4=\{(t,r)\in Q: t< t_0, r\ge r_0\}.
\end{align*}
Then the path space is defined by 
\begin{align*}
D_2&=\Bigl\{f:Q\rightarrow \bR: \forall (t_0,r_0)\in Q,\lim_{\begin{subarray}{l} (t,r)\in Q_{(t_0,r_0)}^i \\ (t,r)\rightarrow (t_0,r_0)\end{subarray}}f(t,r)\mbox{ exists for }i\in\{1,2,3,4\}\\
&\qquad \qquad 
\mbox{and }\; \lim_{\begin{subarray}{l} (t,r)\in Q_{(t_0,r_0)}^1 \\ (t,r)\rightarrow (t_0,r_0)\end{subarray}}f(t,r)=f(t_0,r_0)\Bigr\}.
\end{align*}
 $D_2$ is separable and topologically complete under  a Skorohod-type metric 
$$d(f,g)=\inf_{\lambda}\max(\norm{f-g\circ\lambda}_\infty\,,\norm{\lambda-\text{id}}_\infty ), \quad f,g\in D_2, $$
where $\lambda(t,r)=(\lambda_1(t), \lambda_2(r))$  for    strictly increasing, continuous bijections $\lambda_1$ and $\lambda_2$.   For details we refer to Bickel and Wichura  \cite{bick-wich}.  

%


\begin{theorem}\label{thm2-6}  Assume  $d=1$, \eqref{w-ass},  \eqref{xi-ass},  and 
  $\E[\xi_t(x)^{12}]<\infty$.  
 Assume that the initial increment sequence $\{\eta_0(x)\}_{x\in\Z}$ satisfies one of the assumptions {\rm(}a{\rm)}, {\rm(}b{\rm)},  or {\rm(}c{\rm)}:  
  
 \smallskip  
  
{\rm(a)}    $\{\eta_0(x)\}_{x\in\Z}$ is an  i.i.d.\ sequence and $\mE[\,\eta_0(x)^{12}\,]<\infty$. 
  
 \smallskip

{\rm(b)}    $\{\eta_0(x)\}_{x\in\Z}$ is a strongly mixing stationary sequence,  and  $\exists\delta>0$ such that $\bE[\,\abs{\eta_0(0)}^{12+\delta}\,]<\infty$ and  the strong mixing coefficients in \eqref{eta-a2}  satisfy 
\be\label{eta-a77} \sum_{j=0}^\infty (j+1)^{10+132/\delta}\alpha(j)<\infty. \ee 
   

 {\rm(c)}  $\{\eta_0(x)\}_{x\in\Z}$ has the distribution $\pi_0$ of the sequence 
 in  \eqref{(c)}.  
   
 \smallskip

Then on any rectangle $Q=[0,T]\times [-R,R]$,   $\fluc_n\Rightarrow Z$ on path space  $D_2$.   
\end{theorem}

This completes the description of results and we turn to proofs.

\section{Proofs for invariant distributions} \label{sec:inv-pf}


For $x\in\Z^d$, $i\in[d]$, and  $s<t$ in $\Z$ let 
\[  \Delta_{s,t}(x-e_i, x)= \sum_{k=s+1}^t\sum_{y} \xi_k(y) [ p^{t-k}(x,y )-p^{t-k}(x-e_i,y)]. \]
Then $\Delta_{-\infty, t}$ is what we denoted by $\Delta_t$ in \eqref{De7}.  As $t-s\to\infty$, the random process above indexed by $(x,i)$ converges weakly to the configuration $\Delta_0\sim\pi_0$, by the convergence in \eqref{De7}.   
The basic evolution equation \eqref{eta-11}  can be written as  
\be\label{nu-aux4} \begin{aligned}
\eta_t(x-e_i, x) 
&=\sum_y p^t(x,y)\eta_0(y-e_i, y)  \;+\;  \Delta_{0,t}(x-e_i,x), \qquad t\ge 0. 
\end{aligned}\ee 
The proof of Theorem \ref{thm-pi-u} is contained in the next lemma.  

\begin{lemma}   We have these properties of invariant distributions.  \\[-10pt]

 {\rm (a)}  Let  $\nu\in\cI$ and assume that under $\nu$ each variable  $\eta(x-e_i,x)$  has finite mean and variance. Then 
 \be\label{nu-var1} \Var^\nu[\eta(x-e_i,x)]\ge \Var^{\pi_0}[\eta(x-e_i,x)]. \ee
  Equality in \eqref{nu-var1} holds for all  $i\in[d]$ and $x\in\Z^d$ iff $\nu\in\{\pi_u: u\in\cH_d\}$.  \\[-10pt]

 {\rm (b)}   For each  $u\in\cH_d$,  $\pi_u$ is an extreme point of $\cI$. 
 
\end{lemma}

\begin{proof}  (a)  When the process $\eta_t\sim\nu$  in \eqref{nu-aux4} is stationary, 
taking expectations on both sides shows that $u_i(x)=E^\nu[\eta(x-e_i,x)]$ is harmonic.  Then also the process $\bar\eta_t(x-e_i,x)=\eta_t(x-e_i,x) -u_i(x)$ satisfies \eqref{nu-aux4} and is stationary.  By stationarity, by the independence of the terms on the right of \eqref{nu-aux4},   by time-shift-invariance of $\xi$,  by the  $L^2(\P)$-convergence in \eqref{De7}, and finally by the definition of $\pi_0$: 
\be\label{nu-aux13} \begin{aligned}
&\Var^\nu[\eta(x-e_i,x)] =\mE^\nu[\bar\eta_t(x-e_i,x)^2] \\
&=
\mE^\nu\biggl[ \biggl( \sum_y p^t(x,y)\bar\eta_0(y-e_i, y)\biggr)^2\,\biggr] 
  \;+\;  \E[ \Delta_{0,t}(x-e_i,x)^2\,] \\
  &\ge \E[ \Delta_{-t, 0}(x-e_i,x)^2\,]  \; \underset{t\to\infty}\longrightarrow  \;
     \E[ \Delta_{0}(x-e_i,x)^2\,]  =  \Var^{\pi_0}[\eta(x-e_i,x)].  
\end{aligned}\ee
This gives the inequality claimed in part (a).  

Pick   $x_1,\dotsc,x_N\in\Z^d$,  $i_1,\dotsc,i_N\in[d]$, and  $\alpha_1,\dotsc, \alpha_N\in\R$,  and  repeat \eqref{nu-aux4} for    linear combinations:  
\be\label{nu-aux5} \begin{aligned} \sum_{\ell=1}^N \alpha_\ell \,\bar\eta_t(x_\ell-e_{i_\ell}, x_\ell) &=   \sum_{\ell=1}^N \alpha_\ell  \sum_y p^t(x_\ell,y)\bar\eta_0(y-e_{i_\ell}, y)\\
&\qquad\qquad \qquad 
  \;+\;   \sum_{\ell=1}^N \alpha_\ell  \Delta_{0,t}(x_\ell-e_{i_\ell}, x_\ell) 
  = S_t+D_t. 
\end{aligned}\ee 
Equality in \eqref{nu-var1}   $\forall\,i\in[d], x\in\Z^d$ and \eqref{nu-aux13} imply that $E^\nu[ S_t^2]\to 0$.  Consequently, on the right-hand side of \eqref{nu-aux5}, the first sum vanishes and the second sum converges in distribution  to $\sum_\ell\alpha_\ell  \Delta_{0}(x_\ell-e_{i_\ell}, x_\ell)$. This implies that under $\nu$, the configuration  $\bar\eta$ has the  distribution of the configuration $\Delta_0$, which says that  $\nu$ is among the distributions $\{\pi_u: u\in\cH_d\}$.   

\smallskip 

(b) To get a contradiction, suppose $\pi_u=\beta \nu^1+(1-\beta)\nu^0$ for $\beta\in(0,1)$ and  $\nu^1,\nu^0\in\cI$.    Let $\eta^\ell$ denote the configuration under $\nu^\ell$.  
  Let $J$ be a Bernoulli($\beta$) variable: $P(J=1)=\beta=1-P(J=0)$.  Then  $\nu=\beta \nu^1+(1-\beta)\nu^0$ is the distribution of the configuration $\eta^J$, and by the definition of $\pi_u$,  the assumption $\pi_u=\nu$  implies the distributional equality 
 \be\label{nu-aux18}  u_i(x)+\Delta_0(x-e_i,x)\overset{d}= \eta^J(x-e_i,x), \qquad x\in\Z^d,\,  i\in[d]. \ee
It follows that $\eta^1$ and $\eta^0$ have finite means $u^\ell_i(x)= E^{\nu^\ell}[\eta^\ell(x-e_i,x)]$ and variances, and 
 \be\label{nu-aux19}  u_i(x)=E^\nu[\eta^J(x-e_i,x)]=\beta u^1_i(x)
+ (1-\beta) u^0_i(x). 
\ee
From \eqref{nu-aux18}, by expanding the variance, and then by \eqref{nu-var1}, 
\begin{align*}
&\Vvv[ \Delta_0(x-e_i,x) ]  = E^\nu\bigl[ \bigl(\eta^J(x-e_i,x)-u_i(x)\bigr)^2\,\bigr]\\
&=\beta E^{\nu^1}\bigl[ \bigl(\eta^1(x-e_i,x)-u^1_i(x) +(1-\beta)\{ u^1_i(x)-u^0_i(x)\}  \bigr)^2\,\bigr]\\
&\qquad\qquad  + \ 
(1-\beta)  E^{\nu^0}\bigl[ \bigl(\eta^0(x-e_i,x)-u^0_i(x) - \beta\{u^1_i(x)-u^0_i(x)\} \bigr)^2\,\bigr] \\[2pt] 
&= \beta \Var^{\nu^1}[  \eta^1(x-e_i,x)] + (1-\beta) \Var^{\nu^0}[  \eta^0(x-e_i,x)]  
+\beta(1-\beta) \bigl\{ u^1_i(x)-u^0_i(x) \bigr\}^2 \\
&\ge  \Vvv[ \Delta_0(x-e_i,x) ]  +  \beta(1-\beta) \bigl\{ u^1_i(x)-u^0_i(x) \bigr\}^2.  
\end{align*}
This forces first $u^1=u^2$, so by \eqref{nu-aux19} $u=u^1=u^0$.  Second, we get equality of the variances
\[   \Var^{\nu^1}[  \eta^1(x-e_i,x)] =  \Var^{\nu^0}[  \eta^0(x-e_i,x)]  = \Vvv[ \Delta_0(x-e_i,x) ]   \]
which by part (a) forces  $\nu^\ell=\pi_{u^\ell}$ which equals $\pi_u$.  

\end{proof}

We prove Theorem \ref{eta-thm-2} next because it will be helpful for a later proof.  

\begin{proof}[Proof of Theorem \ref{eta-thm-2}]    
By \eqref{nu-aux4}, the assumption on $\eta_0$,  and harmonicity, 
\be \begin{aligned}
\eta_t(x-e_i,x) 
&=  \sum_{y}p^{t}(x,y) u_i(y)\;+\; \sum_{y}p^{t}(x,y)\zeta(y,i)\;+\; \Delta_{0,t}(x-e_i,x)\\
&=   u_i(x)\;+\; \sum_{y}p^{t}(x,y)\zeta(y,i)\;+\;  \Delta_{0,t}(x-e_i,x). 
\end{aligned} \label{eta88}\ee
The terms on the right are independent.  
As $t\to\infty$,    the second term  on the right  converges in $L^1$ to the constant $  E\zeta(0, i)=0$ by Lemma \ref{lm-fourier} from the appendix.   As a process indexed by $(x,i)$, the last term converges in distribution to $\Delta_0$.  Since the configuration $\{u_i(x)+\Delta_0(x-e_i, x)\}_{i\in[d],\,x\in\Z^d}$ has distribution $\pi_u$, we have shown that $\eta_t$ converges weakly to $\pi_u$. 
  \end{proof} 

For a probability measure $\nu\in\cM_1$,  let 
\[   A(\nu, r)=\max_{x\in\Z^d:\, \abs{x}\le r} E^\nu\abs{\eta(x)}. \] 
For configurations $\eta\in(\R^d)^{\Z^d}$, let 
\be\label{sss-def}   \sss_t(\eta, x, i)= \sum_y p^t(x,y)\eta(y-e_i, y) . \ee

  \begin{lemma} \label{lm-nu-goal}    Let $\nu\in \cM_1$   satisfy $r^{-1/2}A(\nu, r)\to 0$ 
as $r\to\infty$. Then for  $x\in\Z^d$ and  $i,j\in[d]$,  \ \ 
\[  \lim_{t\to\infty}   E^\nu\abs{\sss_t(\eta, x,i)-\sss_t(\eta, x-e_j, i)} = 0 . \]  
\end{lemma}  

\begin{proof}     By Theorem 4 of Gamkrelidze \cite{Gamk-85} (see also the next to last paragraph of \cite{Gamk-85}),    under assumptions \eqref{w-ass} on the kernel,  there is a constant $C<\infty$ such that 
\be\label{gram5}
\sum_{x\in\Z^d}\sum_{\ell=1}^d \abs{p^t(0,x)-p^t(0, x-e_\ell)} \le C t^{-1/2}  
\quad \forall t\in\Z_+.  
\ee
 By \eqref{w-ass}  kernel $p$ has finite range $M$.  So by  \eqref{gram5},   
\begin{align}
E^\nu\abs{\sss_t(\eta, x, i)-\sss_t(\eta, x-e_j, i)}  &\le 
\sum_{y: \abs y\le Mt+1 }  \abs{p^t(0,y)-p^t(0,y+e_j)} \cdot E^\nu\abs{\eta(x+y)}  \nn\\
&\le  Ct^{-1/2}  A(\nu, Mt +\abs x +1)   \nn
\end{align}
and by the  assumption   the last quantity vanishes as $t\to\infty$.   
\end{proof}

 For $\alpha\in\R^d$ let $\pi_\alpha=\pi_u$ for the constant function $u(x)=\alpha$.

\begin{lemma}\label{lm-nu-8}
Let $\nu\in \cI$   satisfy $r^{-1/2}A(\nu, r)\to 0$ 
as $r\to\infty$.   Then there exists a  probability measure $\gamma$ on $\R^d$ such that $\nu=\int \pi_\alpha\,\gamma(d\alpha)$.  
\end{lemma}

\begin{proof}  

Using \eqref{sss-def} write  the stationary evolution \eqref{nu-aux4}  for  $\eta_t\sim\nu$    as   
\be\label{nu-aux6} \begin{aligned}
\eta_t(x-e_i, x)  
&=\sss_t(\eta_0,x,i) +  \Delta_{0,t}(x-e_i, x)  . 
\end{aligned}\ee 

Pick lattice points $x_k\in\Z^d$, indices $i_k\in[d]$,  and  reals $\alpha_k$,  and consider the distributions of the linear combinations 
\begin{align*}
\sum_{k=1}^N \alpha_k  \eta_t(x_k-e_{i_k}, x_k) &= \sum_{k=1}^N \alpha_k  \sss_t(\eta_0,x_k,i_k) +  \sum_{k=1}^N \alpha_k  \Delta_{0,t}(x_k-e_{i_k}, x_k)  = S_t+ D_t
\end{align*}
and the same thing shifted by $-e_j$: 
\begin{align*}
\sum_{k=1}^N \alpha_k  \eta_t(x_k-e_j-e_{i_k}, x_k-e_j) &= \sum_{k=1}^N \alpha_k  \sss_t(\eta_0, x_k-e_j,i_k)\\
&\quad  +  \sum_{k=1}^N \alpha_k  \Delta_{0,t}(x_k-e_j-e_{i_k}, x_k-e_j)  = \wt S_t+ \wt D_t.  
\end{align*}

Compare   $\nu$ and its shift  by $-e_j$ through characteristic functions.  Use  time invariance,  the triangle inequality,  and  $\abs{e^{\iota a}}\le 1$  and $\abs{e^{\iota a}-e^{\iota b}}\le \abs{a-b}$ for real $a,b$. 
\begin{align*}
&\bigl\lvert \, E^\nu[e^{\iota \sum_{k=1}^N \alpha_k  \eta(x_k-e_{i_k},\, x_k) }]  - E^\nu[e^{\iota \sum_{k=1}^N \alpha_k  \eta(x_k-e_j-e_{i_k},\, x_k-e_j) }] \,\bigr\rvert  \\  
&= 
 \bigl\lvert \, \mE^\nu[e^{\iota \sum_{k=1}^N \alpha_k  \eta_t(x_k-e_{i_k},\, x_k) }]  - \mE^\nu[e^{\iota \sum_{k=1}^N \alpha_k  \eta_t(x_k-e_j-e_{i_k},\, x_k-e_j) }] \,\bigr\rvert  \\
&=  \bigl\lvert \, \mE^\nu[e^{\iota S_t }] \,\E[e^{\iota D_t }]  - \mE^\nu[e^{\iota \wt S_t }] \,\E[e^{\iota \wt D_t }] \,\bigr\rvert  \\
&\le   \bigl\lvert \, \mE^\nu[e^{\iota S_t }]    - \mE^\nu[e^{\iota \wt S_t }]   \,\bigr\rvert  
+     \bigl\lvert \,  \E[e^{\iota D_t }]  -  \E[e^{\iota \wt D_t }] \,\bigr\rvert  \\
&\le \mE^\nu\abs{S_t-\wt S_t} +  
 \bigl\lvert \,  \E[e^{\iota D_t }]  -  \E[e^{\iota \wt D_t }] \,\bigr\rvert .  
\end{align*}
The last line above vanishes as $t\to\infty$  by Lemma \ref{lm-nu-goal} and because process $\Delta_{0,t}$ converges weakly to process $\Delta_0$ which is invariant under spatial shifts.  
 We have now shown $\nu$ invariant under spatial shifts.  
 
Let $\nu=\int \mu\,\Gamma(d\mu)$ be the ergodic decomposition of $\nu$. $\Gamma$ is a probability measure supported on  probability measures $\mu\in\cM_1$  that are  invariant and  ergodic  under the spatial shift group and that have a finite mean $\alpha(\mu)=E^\mu[\eta(x)]$.  
 
 Let $f$ be a bounded continuous function on $(\R^d)^{\Z^d}$.  By invariance of $\nu$ and Theorem \ref{eta-thm-2}, 
 \begin{align*}
 \int f\,d\nu &=   \mE^\nu[f(\eta_t)] =  \int  \mE^\mu[f(\eta_t)]\,\Gamma(d\mu)
  \; \underset{t\to\infty}\longrightarrow  \;     \int  E^{\pi_{\alpha(\mu)}}(f)\,\Gamma(d\mu) 
  = \int  E^{\pi_{\alpha}}(f)\,\gamma(d\alpha)
 \end{align*}
 where $\gamma$ is the distribution of the mean of $\mu$ under $\Gamma(d\mu)$.  
   \end{proof} 


\begin{proof}[Proof of Theorem \ref{thm-uniq}]   By harmonicity of $u$,  the process $\wt\eta_t(x-e_i,x)=\eta_t(x-e_i,x)- u_i(x)$  is also time-stationary.  By assumption \eqref{nu-mom} its marginal distribution $\wt\nu\in\cI$  satisfies the growth bound  $r^{-1/2} A(\wt\nu, r)\to 0$.    Lemma \ref{lm-nu-8} applied to $\wt\nu$ gives the distributional identity 
$  \wt\eta(x) \overset{d}=  X+\Delta_0(x)  $ where $\wt\eta\sim\wt\nu$ and   $X\in\R^d$ is a 
  $\wt\gamma$-distributed  random vector  independent of $\Delta_0$.    Consequently $\bar u(x)=E^\nu[\eta(x)]=u(x) +E^{\wt\gamma}[X] $, and the claim about substituting $\bar u(x)$ into \eqref{nu-mom} follows.   
  
  An application of  Lemma \ref{lm-nu-8}  to the distribution $\bar\nu$ of the configuration 
$\bar\eta_t(x)=\eta_t(x)- \bar u(x)$ gives a random $Y\in\R^d$ such that $\eta\sim\nu$ satisfies 
$ \{\eta(x) \}_{x\in\Z^d} \overset{d}=  \{ \bar u(x)+Y+ \Delta_0(x)  \}_{x\in\Z^d} $.  With $Y\sim\gamma$, this is the same as $\nu=\int \pi_{\bar u+\alpha}\,\gamma(d\alpha)$.  
 \end{proof}

\begin{proof}[Proof of Theorem \ref{thm-cov}]  Identity \eqref{cov=} says that $V_0(0,x)$ is a Fourier coefficient of (a constant multiple of)  the function 
$f(\theta)=\frac{1-\cos \theta}{1-\phi_q(\theta)}$ on  $ [-\pi,\pi]$.  Extend   $f$ from  $ [-\pi,\pi]$ to a meromorphic function $f(z)=\frac{1-\cos z}{1-\phi_q(z)}$ on the complex plane.  The only pole in some  neighborhood of  $ [-\pi,\pi]$ is $z=0$.   Expansions give   $f(z)=\frac{z^2/2 + O(z^4)}{\sigma_1^2z^2 + O(z^4)}$  and we see that   $z=0$ is a removable singularity.    Consequently  $f$ is analytic on $[-\pi,\pi]$ and its   Fourier coefficients decay exponentially    (Prop.~1.2.20 on p.~20 of \cite{pinsky-fourier}).  

Identity \eqref{cov-sum} now follows   from 
\be
\lim_{x\to \pm\infty}\bigl[a(x+k)-a(x)\bigr]=\pm\frac{k}{2\sigma_1^2}  \label{l38a}\ee
(P29.2 on p.~354 of  \cite{spitzer}). Note that the variance of the $q$-kernel is $2\sigma_1^2$  \eqref{p-6}.  
\end{proof}

\begin{proof}[Proof of  part {\rm(a)} of Theorem \ref{h-thm}]
  By square-integrability  \eqref{xi-ass} and Theorem 3.3.8 in \cite{durr},   the characteristic function $\varphi(\alpha)=\E(e^{\iota\alpha\xi_0(0)})$ satisfies 
 \[  \varphi(\alpha)=1-\tfrac12 \alpha^2\sigma_\xi^2 +o(\alpha^2) \qquad 
 \text{as } \  \alpha\to 0. \] 
 
Suppose $\{h_t\}_{t\in\Z_+}$ is a time-stationary  $\R^{\Z^d}$-valued Markov chain that satisfies \eqref{i1}.   Iterating   \eqref{i1} gives  
\be\begin{aligned}
&\h_{t}(x)=\sum_{y\in\Z^d}  p^t(x,y)\h_0(y) \, +\, \  \sum_{k=0}^{t-1}\sum_{y\in\Z^d} p^k(x,y) \xi_{t-k}(y)  .
\end{aligned} \ee
On the right-hand side  the initial profile $\h_0$ is independent of the i.i.d.\ variables $\{\xi_j(y)\}$.  By this independence, by the fact that characteristic functions are bounded in absolute value  by 1,  and by the invariance $\h_0\overset{d}=\h_t$,  we have, for any $0<s<t$,  
\begin{align*}  
&\absb{\,\mE[ e^{\iota\alpha \h_0(x)}]  \,}  \le  \absb{\, \E[e^{ \iota\alpha  \sum_{k=s}^{t-1}\sum_{y\in\Z^d} p^k(x,y) \xi_{t-k}(y)  } ] \,} 
= \prod_{k=s}^{t-1}\prod_{y\in\Z^d}  \absb{ \varphi\bigl(\alpha p^k(x,y)\bigr)} \\
&\qquad 
= \prod_{k=s}^{t-1}\prod_{y\in\Z^d}  \Bigl(   1- \tfrac12 \alpha^2 \sigma_\xi^2 \bigl(p^k(x,y)\bigr)^2  \bigl(1 + o(1)\bigr)  \Bigr) \\
&\qquad
\le \exp\biggl\{   - \tfrac12 \alpha^2 \sigma_\xi^2 \bigl(1 + o(1)\bigr)  \sum_{k=s}^{t-1}\sum_{y\in\Z^d}   \bigl(p^k(x,y)\bigr)^2     \biggr\}
\le     \exp\biggl\{   - c \alpha^2  \sum_{k=s}^{t-1}q^k(0,0)      \biggr\}    
\end{align*}
with a constant  $c>0$.  
The second equality above is justified by fixing $s$ large enough and by   the dimension-independent  uniform bound $p^k(x,y)\le Cs^{-1/2} $ for $k\ge s$
({P7.6}  on p.~72 in \cite{spitzer}).    In dimensions $d\in\{1,2\}$ the $q$-walk is recurrent (T8.1 on p.~83 in \cite{spitzer}), and consequently taking $t\to\infty$  above 
gives  $\mE[ e^{\iota\alpha \h_0(x)}]=0$ for $\alpha\ne 0$.     This contradiction with the continuity of a  characteristic function at $\alpha=0$ shows that in  $d\in\{1,2\}$  there can exist no time-stationary  height process.  
\end{proof}

\section{Height fluctuations:  limits of finite-dimensional distributions}\label{sec:fd}

 Let $\{X^i_t\}_{t\in\Z_+}$ denote a  random walk on $\Z$  with initial point $X_0^i=i$ and  transition probability  $p(x,y)$ of \eqref{w-ass},  and let $\{Y^i_t\}_{t\in\Z_+}$ similarly denote a random walk on $\Z$  that uses transition  $q(x,y)$ of \eqref{21-d2}.  Equivalently,   $Y_t^i=\tilde{X}_t^{i}-X_t^{0}$  for two independent $p$-walks 
$\tilde{X}_\centerdot^{i}$ and $X_\centerdot^{0}$.   Under assumption \eqref{w-ass}, the $q$-walk also has   span 1.  

 Probabilities and expectations of these walks are denoted by $P$ and $E$,  always taken under {\it fixed}   $\eta_0$ and $\xi$.  In particular, we have the following ``dual'' representation of the harness process:    for   $t\in\bZ_+$ and  $i\in\bZ$,
\be    h_t(i)= E\bigl[\h_0(X_t^{i})\bigr]+\sum_{k=1}^t E \bigl[\xi_k(X_{t-k}^{i})\bigr], \label{21-l1}\ee 
where we emphasize that $h_0$ and $\xi$ have {\it not} been averaged over on  the right because the $E$ acts only on the walk $X^i_\centerdot$.

This lemma is a consequence  of  the  local CLT (Theorem 3.5.2 in \cite{durr}) and will be used several times in the sequel.  
\begin{lemma}\label{crllyB-1}
For a mean 0, span 1 random walk $S_n$ on $\bZ$ with finite variance $\sigma^2$, $a\in \bR$,  and points $a_n\in\bZ$ such that  $\lim_{n\to\infty}a_n/\sqrt{n}=a$, we have 
\be
\lim_{n\to\infty}\frac{1}{\sqrt{n}}\sum_{k=0}^{\lfloor nt\rfloor-1} P(S_k=a_n)=\frac{1}{\sigma^2}\int_0^{\sigma^2t}\frac{1}{\sqrt{2\pi v}}\exp\Bigl(-\frac{a^2}{2v}\Bigr)dv.\label{2-l11}\ee
\end{lemma}


%
%
%
%
%


The analysis of height fluctuations  begins with a decomposition of     the scaled height function  as  
\be
\fluc_n(t,r)=\mu_0\overline{H}_n(t,r)+\overline{F}_n(t,r)+\overline{S}_n(t,r)\label{2-7}
\ee
where, with   $y(n)=\lfloor ntb\rfloor +\lfloor r\sqrt{n}\rfloor$,  
\begin{align}
\overline{H}_n(t,r)&=n^{-1/4}\bigl( E  \bigl[X_{\lfloor nt\rfloor}^{y(n)}\bigr]-r\sqrt{n}\,\bigr), \label{2-4}\\
\overline{F}_n(t,r)&=n^{-1/4}\sum_{k=1}^{\lfloor nt\rfloor}\sum_{x\in\mathbb{Z}}\xi_k(x) P \bigl(X_{\lfloor nt\rfloor -k}^{y(n)}=x\bigr)=n^{-1/4}\sum_{k=1}^{\lfloor nt\rfloor} E \bigl[ \xi_k(X_{\lfloor nt\rfloor -k}^{y(n)})\bigr] ,\label{2-6}\\
\overline{S}_n(t,r)&=n^{-1/4}\sum_{i\in\mathbb{Z}}\bigl(\eta_0(i)-\mu_0\bigr)\left\{\ind_{\{i>0\}} P \bigl(X_{\lfloor nt\rfloor}^{y(n)}\ge i\bigr)-\ind_{\{i\le 0\}} P \bigl(X_{\lfloor nt\rfloor}^{y(n)}< i\bigr)\right\}.  \label{2-5}
\end{align}
  Recall that   $b=-\pmu=- E  [X^0_1]$. 
 \eqref{2-7} follows from the random walk representation 
 \be
\fluc_n(t,r)=n^{-1/4}\biggl\{ E \bigl[\h_0(X_{\lfloor nt\rfloor}^{y(n)})\bigr]+\sum_{k=1}^{\lfloor nt\rfloor} E\bigl[ \xi_k(X_{\lfloor nt\rfloor -k}^{y(n)})\bigr] -\mu_0r\sqrt{n}\biggr\}\label{2-l1}\ee
from \eqref{21-l1},   $\h_0(0)=0$,  $\eta_0(k)=\h_0(k)-\h_0(k-1)$, and 
 from 
\begin{align*}
&  E [\h_0(X_{\lfloor nt\rfloor}^{y(n)})]-\mu_0r\sqrt{n}\\
&=  E \biggl[\ind_{\{X_{\lfloor nt\rfloor}^{y(n)}>0\}}\sum_{i=1}^{X_{\lfloor nt\rfloor}^{y(n)}}\eta_0(i)-\ind_{\{X_{\lfloor nt\rfloor}^{y(n)}<0\}}\sum_{i=X_{\lfloor nt\rfloor}^{y(n)}+1}^0\eta_0(i)\biggr]-\mu_0r\sqrt{n}\\
&=\sum_{i>0}\eta_0(i) P  (X_{\lfloor nt\rfloor}^{y(n)}\ge i)-\sum_{i\le 0}\eta_0(i) P (X_{\lfloor nt\rfloor}^{y(n)}< i)-\mu_0r\sqrt{n}\\
&=n^{1/4}\mu_0\overline{H}_n(t,r)+n^{1/4}\overline{S}_n(t,r).
\end{align*}

The three terms in \eqref{2-7} will be treated separately both for the convergence of finite-dimensional distributions (this section)  and process-level tightness (Section \ref{sec:tight}).  

Begin by observing that the   first term $\mu_0\overline{H}_n(t,r)$  on the right of \eqref{2-7} is irrelevant:   
since $E(X_{\lfloor nt\rfloor}^{y(n)})=\pmu\lfloor nt\rfloor +y(n)=-b\lfloor nt\rfloor+\lfloor ntb\rfloor +\lfloor r\sqrt{n}\rfloor$,   we have 
 $\overline{H}_n(t,r)=O(n^{-1/4})$  and  
 $\lim_{n\to\infty}\mu_0\overline{H}_n(t,r)=0$, uniformly over $(t,r)$.  
 
  Next note that $\overline{F}_n$ and $\overline{S}_n$ are independent.  $\overline{F}_n$ depends only on $\xi$ and,  in the limit,  furnishes the first   term in \eqref{2-t13b}  and   \eqref{Z-int}.  
  $\overline{S}_n$ depends only on $\eta_0$  and  gives  in the limit the second   term in \eqref{2-t13b}  and   \eqref{Z-int}.     The weak limits of  $\overline{F}_n$ and $\overline{S}_n$ are treated separately in Propositions \ref{lmm3-10} and \ref{lmm3-8}.
 
 \medskip 
 
 We begin with   $\overline{F}_n$. 
   Let   $\{F(t,r): t\in\bR_+,r\in\bR\}$ be a   mean-zero Gaussian process with  covariance 
\be
 {E}\bigl[F(t,r)F(s,q)\bigr]=\frac{\sigma_\xi^2}{\sigma_1^2}\Gamma_1\bigl((t,r),(s,q)\bigr).\label{2-t9a}\ee
Here are  alternative expressions for $\Gamma_1$  \cite[Chapter 2]{sepp-10-ens}: 
\be\begin{aligned}
\Gamma_1\bigl((s,q),(t,r)\bigr)&=\int_{-\infty}^\infty \bigl[ P (B_{\sigma_1^2s}\le q-x) P (B_{\sigma_1^2t}>r-x)\\
&\qquad\qquad 
- P (B_{\sigma_1^2s}\le q-x,B_{\sigma_1^2t}>r-x)\bigr]\,\mathrm{d}x,
\end{aligned} \label{2-10a}\ee
where $B_t$ is a standard 1-dimensional Brownian motion,  and
\be\Gamma_1\bigl((s,q),(t,r)\bigr)=\frac{1}{2}\int_{\sigma_1^2|t-s|}^{\sigma_1^2(t+s)}\frac{1}{\sqrt{2\pi v}}\exp\bigl\{-\frac{1}{2v}(r-q)^2\bigr\}dv.\label{2-10b}\ee

 \begin{proposition}\label{lmm3-10}
 Assume \eqref{xi-ass} and   $\E[\xi_t(x)^4]<\infty$. 
 Then,   as $n\to\infty$,  the finite-dimensional distributions of the process $\overline{F}_n$   converge weakly    to  those of  $F$. 
\end{proposition}

\begin{proof}
First we argue the convergence of the covariance of $\overline{F}_n$.  
  Let  $X_\centerdot^{\lfloor ntb\rfloor +\lfloor r\sqrt{n}\rfloor}$ and $X_\centerdot^{\lfloor nsb\rfloor +\lfloor q\sqrt{n}\rfloor}$ denote two independent random walks with transition probability $p$ (even if $(t,r)$ and $(s,q)$ should happen to  coincide).


For $s=t$ and  $r, q\in\bR$, set  $x_n=\lfloor r\sqrt{n}\rfloor -\lfloor q\sqrt{n}\rfloor$. Then,
\begin{align*}
&\mathbb{E}\bigl[\,\overline{F}_n(t,r)\overline{F}_n(t,q)\bigr]=n^{-1/2}\sum_{k=1}^{\lfloor nt\rfloor}\sum_{x\in\mathbb{Z}}\sigma_\xi^2 P \bigl(X_{\lfloor nt\rfloor -k}^{\lfloor ntb\rfloor +\lfloor r\sqrt{n}\rfloor}=x\bigr) P \bigl(X_{\lfloor nt\rfloor -k}^{\lfloor ntb\rfloor +\lfloor q\sqrt{n}\rfloor}=x\bigr)\\
&=n^{-1/2}\sigma_\xi^2\sum_{k=1}^{\lfloor nt\rfloor}q^{\lfloor nt\rfloor -k}(x_n,0)=n^{-1/2}\sigma_\xi^2\sum_{k=0}^{\lfloor nt\rfloor-1}q^k(0,x_n)\\
&\underset{n\to\infty} \longrightarrow  \ \frac{\sigma_\xi^2}{2\sigma_1^2}\int_0^{2\sigma_1^2t}\frac{1}{\sqrt{2\pi v}}\exp\Bigl\{-\frac{(r-q)^2}{2v}\Bigr\}dv
\;=\; \frac{\sigma_\xi^2}{\sigma_1^2}\Gamma_1\bigl((t,q),(t,r)\bigr).
\end{align*}
The limit came from  \eqref{2-l11}. 

%

Next take   $s<t$ and set $X_n'=X_{\lfloor nt\rfloor-\lfloor ns\rfloor}^{\lfloor ntb\rfloor +\lfloor r\sqrt{n}\rfloor}-\lfloor nsb\rfloor -\lfloor q\sqrt{n}\rfloor$. By  the Markov property, 
\begin{align*}
&\mathbb{E}\bigl[\,\overline{F}_n(t,r)\overline{F}_n(s,q)\bigr]=n^{-1/2}\sigma_\xi^2\sum_{k=1}^{\lfloor ns\rfloor} P \bigl(X_{\lfloor nt\rfloor-k}^{\lfloor ntb\rfloor +\lfloor r\sqrt{n}\rfloor}-X_{\lfloor ns\rfloor-k}^{\lfloor nsb\rfloor +\lfloor q\sqrt{n}\rfloor}=0\bigr)\\
&=n^{-1/2}\sigma_\xi^2 E \biggl[\; \sum_{k=1}^{\lfloor ns\rfloor} P \bigl(X_{\lfloor nt\rfloor-k}^{\lfloor ntb\rfloor +\lfloor r\sqrt{n}\rfloor}-X_{\lfloor ns\rfloor-k}^{\lfloor nsb\rfloor +\lfloor q\sqrt{n}\rfloor}=0\; \big\vert\; X_{\lfloor nt\rfloor-\lfloor ns\rfloor}^{\lfloor ntb\rfloor +\lfloor r\sqrt{n}\rfloor}\bigr)\biggr]\\
&=n^{-1/2}\sigma_\xi^2 E \biggl[\;\sum_{k=1}^{\lfloor ns\rfloor}q^{\lfloor ns\rfloor-k}\left(X_{\lfloor nt\rfloor-\lfloor ns\rfloor}^{\lfloor ntb\rfloor +\lfloor r\sqrt{n}\rfloor}-\lfloor nsb\rfloor -\lfloor q\sqrt{n}\rfloor\,,\, 0\right)\biggr] \\
&=n^{-1/2}\sigma_\xi^2 \sum_{k=0}^{\lfloor ns\rfloor-1}  E  [\,q^{k}(0,X_n')\,] .
\end{align*}
 
We use  again   Lemma  \ref{crllyB-1}   to derive the limit. By the CLT,   $n^{-1/2}X_n'\Rightarrow B_{\sigma_1^2|t-s|}+(r-q)$. 
By a standard construction (Theorem 3.2.2 in \cite{durr}), we can  find random variables  $\hat X_n\overset{d}=X_n'$ such that $n^{-1/2}\hat X_n\to B_{\sigma_1^2|t-s|}+(r-q)$ almost surely.  
Then by \eqref{2-l11},
$$\lim_{n\to\infty}n^{-1/2}\sum_{k=0}^{\lfloor ns\rfloor-1}q^{k}(0,\hat X_n)=\frac{1}{2\sigma_1^2}\int_0^{2\sigma_1^2s}\frac{1}{\sqrt{2\pi v}}\exp\Bigl\{-\frac{(B_{\sigma_1^2|t-s|}+r-q)^2}{2v}\Bigr\}dv\quad \text{a.s.}$$
By the uniform bound $q^k(0,y)\le Ck^{-1/2}$  we can apply dominated convergence.  
\be\begin{aligned}
&\lim_{n\to\infty}\mathbb{E}\bigl[\,\overline{F}_n(t,r)\overline{F}_n(s,q)\bigr]=\lim_{n\to\infty}n^{-1/2}\sigma_\xi^2 E \biggl[\sum_{k=0}^{\lfloor ns\rfloor-1}q^{k}(0,\hat X_n)\biggr]\\
&=\sigma_\xi^2 E \biggl[\frac{1}{2\sigma_1^2}\int_0^{2\sigma_1^2s}\frac{1}{\sqrt{2\pi v}}\exp\Bigl\{-\frac{(B_{\sigma_1^2|t-s|}+r-q)^2}{2v}\Bigr\}dv\biggr]\\
&=\frac{\sigma_\xi^2}{2\sigma_1^2}\int_{\sigma_1^2|t-s|}^{\sigma_1^2(t+s)}\frac{1}{\sqrt{2\pi v}}\exp\Bigl\{-\frac{(r-q)^2}{2v}\Bigr\}dv \;= \; \frac{\sigma_\xi^2}{\sigma_1^2}\Gamma_1\bigl((t,r),(s,q)\bigr).
\end{aligned}  \label{F-1}\ee
The second  last equality comes from a convolution of Gaussians.   
 


Having taken care of the covariance, we turn to  derive the weak limit.  Fix  $N\in\bN$, $(t_j,r_j)\in \bR_+\times \bR$ and $\theta_j\in \bR$ for $j=1,\ldots,N$,  and arrange the indices so   that $0=t_0\le t_1\le t_2\le\cdots\le t_N$.  Then, with   $\ell(k)=i$ iff $\lfloor nt_i\rfloor+1\le k \le \lfloor nt_{i+1}\rfloor$,  
\begin{align*}
&\sum_{j=1}^N\theta_j \overline{F}_n(t_j,r_j)
=n^{-1/4}\sum_{j=1}^N\theta_j\sum_{k=1}^{\lfloor nt_j\rfloor}\sum_{x\in\bZ}\xi_k(x)p^{\lfloor nt_j\rfloor -k}\left(\lfloor nt_jb\rfloor +\lfloor r_j\sqrt{n}\rfloor,x\right)\\
&\quad  =\sum_{\ell=0}^{N-1}\sum_{k=\lfloor nt_{\ell}\rfloor+1}^{\lfloor nt_{\ell+1}\rfloor}n^{-1/4}\sum_{j=\ell+1}^N\theta_j\sum_{x\in\bZ}\xi_k(x)p^{\lfloor nt_j\rfloor -k}\left(\lfloor nt_jb\rfloor +\lfloor r_j\sqrt{n}\rfloor,x\right)\\
&\quad  =\sum_{k=1}^{\lfloor nt_{N}\rfloor}n^{-1/4}\sum_{j=\ell(k)+1}^N\theta_j\sum_{x\in\bZ}\xi_k(x)p^{\lfloor nt_j\rfloor -k}\left(\lfloor nt_jb\rfloor +\lfloor r_j\sqrt{n}\rfloor,x\right)  
\;=\; \sum_{k=1}^{\lfloor nt_{N}\rfloor}  V_{n,k}
\end{align*}
where the last equality defines variables $V_{n,k}$, independent for a fixed $n$. 
We apply the Lindeberg-Feller theorem (Theorem 3.4.5 in \cite{durr})  to $\{V_{n,k}\}$. 
 \eqref{F-1} gives 
\be\label{ff-4} \lim_{n\to\infty}\sum_{k=1}^{\lfloor nt_{N}\rfloor}\bE V_{n,k}^2=\sum_{i,j=1}^N\theta_i \theta_j\frac{\sigma_\xi^2}{\sigma_1^2}\Gamma_1\bigl((t_i,r_i),(t_j,r_j)\bigr).\ee

As preparation for the  second condition of Lindeberg-Feller (the negligibility condition (ii) of Theorem 3.4.5 on p.~129 in \cite{durr})  we estimate a moment.  Abbreviate  $p_j(x)=p^{\lfloor nt_j\rfloor -k}\left(\lfloor nt_jb\rfloor +\lfloor r_j\sqrt{n}\rfloor,x\right)$.  Use    $\bE[\xi_n(k)]=0$,   $\E[\xi_t(x)^4]<\infty$ and   $q^k(x,y)\le Ck^{-1/2}$. 
\be\label{ff-5}\begin{aligned}
\bE[\,\abs{V_{n,k}}^{4} \,] \;  &\le\; \frac{C}{n} \sum_{j=\ell(k)+1}^N  
 \bE\biggl[ \Bigl( \,\sum_{x\in\bZ}\xi_k(x)p_j(x)\Bigr)^4\,\biggr]\\
&\le  \;\frac{C}{n}  \sum_{j=\ell(k)+1}^N \Bigl(\,   \sum_{x,y\in\bZ}  p_j(x)^2 p_j(y)^2
+  \sum_{x\in\bZ}  p_j(x)^4\Bigr) \\
&\le \;\frac{C}{n}  \sum_{j=\ell(k)+1}^N \bigl[ q^{\lfloor nt_j\rfloor -k}\left(0,0\right)\bigr]^2
\;\le\; \frac{C}{n}  \sum_{j=\ell(k)+1}^N  \frac{1}{(\lfloor nt_j\rfloor -k)\vee 1}. 
\end{aligned}\ee
By H\"older's and Chebyshev's inequalities and by \eqref{ff-5}, for  $\e>0$,
\be\label{ff-6}\begin{aligned}
&\sum_{k=1}^{\lfloor nt_{N}\rfloor}\bE\left[V_{n,k}^2 {\bf 1}\{|V_{n,k}|\ge\e\}\right]
\;\le\; 
\e^{-2}\sum_{k=1}^{\lfloor nt_{N}\rfloor} \bE[\,\abs{V_{n,k}}^{4} ] 
\;\le\;   \frac{C}{n\e^2}   \sum_{j=1}^N  \sum_{k=1}^{\lfloor nt_{j}\rfloor} \frac{1}{(\lfloor nt_j\rfloor -k)\vee 1}\\
&\qquad  \le\;   \frac{C}{n\e^2}   (1+\log n)  \;\longrightarrow\; 0  \quad \text{as} \; \;  n\to\infty.  
\end{aligned} \ee


  \eqref{ff-4} and \eqref{ff-6} verify the hypotheses of  the  Lindeberg-Feller  theorem, which implies 
\begin{align*}
\sum_{j=1}^N\theta_j \overline{F}_n(t_j,r_j)  
= \sum_{k=1}^{\lfloor nt_{N}\rfloor} V_{n,k} 
\;\Rightarrow \;    \sum_{j=1}^N\theta_j F(t_j,r_j). 
\end{align*}
The proof for Proposition  \ref{lmm3-10} is complete.
\end{proof}

We turn to   $\overline{S}_n$.  The hypotheses on the initial increments are now relevant.  
Let $\{S(t,r):t\in\bR_+,r\in\bR\}$ be the  mean-zero Gaussian process with   covariance 
\be E\bigl[S(t,r)S(s,q)\bigr]=\varrho_0\Gamma_2\bigl((t,r),(s,q)\bigr).\label{2-t6a}\ee
The function $\Gamma_2$ has this   alternative formulation  \cite[Chapter 2]{sepp-10-ens}:
\be\begin{aligned} 
\Gamma_2\bigl((s,q),(t,r)\bigr)&=\int_{-\infty}^0 P (B_{\sigma_1^2s}>q-x) P (B_{\sigma_1^2t}>r-x)\,\mathrm{d}x\\
&\qquad +\; 
\int_0^\infty  P (B_{\sigma_1^2s}\le q-x) P (B_{\sigma_1^2t}\le r-x)\,\mathrm{d}x \end{aligned} \label{2-r7}\ee
where $B_t$ is a standard 1-dimensional Brownian motion.

\begin{proposition}\label{lmm3-8}
Under the conditions in Theorem \ref{thm2-4},   as $n\to\infty$,  the finite-dimensional distributions of the process $\overline{S}_n$   converge weakly    to  those of  $S$. 
\end{proposition}

\begin{proof}   We have three cases in Theorem \ref{thm2-4}. 

\smallskip

Case (a).  The proof of the lemma under the i.i.d.\  and second moment assumption   goes  via the  Lindeberg-Feller theorem  and can be found in \cite{zhai-phd}.  

\smallskip

Case (b).  
Now  assume that the initial increments $\eta_0=\{\eta_0(x)\}_{x\in\bZ}$ are a stationary, strongly mixing sequence such that,  for some $\delta>0$,    $\mE|\eta_0(0)|^{2+\delta}<\infty$  and  the strong mixing coefficients  $\{\alpha(j)\}$ of $\eta_0$ satisfy $\sum_{j=0}^\infty (j+1)^{2/\delta}\alpha(j)<\infty$. 
  
  We   show the distributional convergence 
\be\label{S-goal} \sum_{j=1}^N\theta_j \overline{S}_n(t_j,r_j)\Rightarrow \sum_{j=1}^N\theta_j S(t_j,r_j)\ee 
for a fixed vector of time-space points  $\{(t_j,r_j)\}_{1\le j\le N}\in (\bR_+\times \bR)^N$ and a fixed real vector   $\bar\theta=(\theta_1,\dotsc, \theta_N)$.  
Rewrite the liner combination  as 
\be\sum_{j=1}^N\theta_j \overline{S}_n(t_j,r_j)=\sum_{i\in\mathbb{Z}}a_{n,i}\bigl(\eta_0(i)-\mu_0\bigr) \label{mid-2b}\ee
where 
\be\begin{aligned}  a_{n,i}=n^{-1/4}\biggl\{{\bf 1}_{\{i>0\}}&\sum_{j=1}^N\theta_j P \bigl(X_{\lfloor nt_j\rfloor}^{\lfloor nt_jb\rfloor +\lfloor r_j\sqrt{n}\rfloor}\ge i\bigr)\\
&\qquad -\; 
{\bf 1}_{\{i\le 0\}}\sum_{j=1}^N\theta_j P \bigl(X_{\lfloor nt_j\rfloor}^{\lfloor nt_jb\rfloor +\lfloor r_j\sqrt{n}\rfloor}<i\bigr)\biggr\}.
\end{aligned} \label{mid-2}\ee


Our first task is the $n\to\infty$  limit of  the variance  
\be\begin{aligned} 
\bar{\sigma}_n^2 &= \mVar\Bigl[ \,\sum_{i\in \bZ}a_{n,i}\bigl(\eta_0(i)-\mu_0\bigr)\Bigr] =\sum_{j,k\in\bZ}a_{n,j}a_{n,k}\mCov (\eta_0(j),\eta_0(k))\\
&=\sum_{\ell\in \bZ}\mCov (\eta_0(0),\eta_0(\ell))\sum_{k\in\bZ}a_{n,k}a_{n,\ell+k}. 
\end{aligned} \label{slim-9}\ee

 The next  lemma   is part of Theorem 1.1 in Rio \cite{rio-lect}.   The quantile function  $Q_X$  of $|X|$  is defined by 
\[  Q_X(u)=\inf\{x\in\bR_+:  P (|X|>x)\le u\},  \qquad 0\le u\le 1. \]

\begin{lemma}\label{covariance-bound}
Let  $X$,  $Y$,  and  $XY$ be  integrable random variables   and let  $\alpha=\alpha\left(\sigma(X),\sigma(Y)\right)$. Then
\be |\Cov(X,Y)|\le 4\int_0^\alpha Q_X(u)Q_Y(u)\,du.\ee
\end{lemma}

We apply this lemma to  show  that the series of covariances $\sum_{k\in\bZ}\mCov[\eta_0(0),\eta_0(k)]$ is absolutely convergent.
Let $Q_{\eta}(u)$ denote the quantile function of $\abs{\eta_0(0)-\mu_0}$. Then
\begin{align}
&\sum_{\ell\in\bZ}\left|\mCov (\eta_0(0), \eta_0(\ell))\right|\le 4\sum_{\ell\in\bZ}\int_0^{\alpha(|\ell|)} Q_{\eta}(u)^2du\le 4\int_0^1 \sum_{\ell\in\bZ}\ind_{\{u\le\alpha(|\ell|)\}} Q_{\eta}(u)^2du\notag\\
&\le  4 \left[\int_0^1 \Bigl(\,\sum_{\ell\in\bZ}\ind_{\{u\le\alpha(|\ell|)\}}\Bigr)^{(2+\delta)/\delta}du\right]^{\delta/(2+\delta)}\left[\int_0^1Q_{\eta}(u)^{2+\delta}du\right]^{2/(2+\delta)}.\notag
\end{align}
Since  $\alpha(n)\searrow 0$ as $n\to\infty$, we have 
\be\label{aux6}\begin{aligned}
\int_0^1 \Bigl(\sum_{\ell\in\bZ}\ind_{\{u\le\alpha(|\ell|)\}}\Bigr)^{(2+\delta)/\delta}du
&=\sum_{j=0}^\infty \int_{\alpha(j+1)}^{\alpha(j)} (2j+1)^{(2+\delta)/\delta}du\\
&=\sum_{j=0}^\infty (2j+1)^{(2+\delta)/\delta}[\alpha(j)-\alpha(j+1)].
\end{aligned}\ee
By summation by parts,  by  $\alpha(n)\ge 0$, and by the summability assumption on $\{\alpha(j)\}$, 
\begin{align*}
&\sum_{j=0}^n (2j+1)^{(2+\delta)/\delta}[\alpha(j)-\alpha(j+1)]\\
&= \alpha(0)-(2n+1)^{(2+\delta)/\delta}\alpha(n+1)+\sum_{j=1}^{n}\left[(2j+1)^{(2+\delta)/\delta}-(2j-1)^{(2+\delta)/\delta}\right]\alpha(j)\\
&\le \alpha(0)+C\sum_{j=1}^{n} (j+1)^{2/\delta}\alpha(j)\le \alpha(0)+C\sum_{j=1}^{\infty} (j+1)^{2/\delta}\alpha(j) <\infty.
\end{align*}
Consequently the quantity in \eqref{aux6} is finite.  

For  a uniform variable $U$  on $(0,1)$,   $Q_{\eta}(U)\overset{d}=\abs{\eta_0(0)-\mu_0}$ and so 
$$\int_0^1\left(Q_{\eta}(u)\right)^{2+\delta}du=\mE\bigl[\,|\eta_0(0)-\mu_0|^{2+\delta}\,\bigr]<\infty.$$

We have shown that 
\be\begin{aligned} 
& \sum_{\ell\in\bZ}\left|\mCov (\eta_0(0), \eta_0(\ell))\right| 
\\
&\qquad \qquad 
\le C\biggl(\,\sum_{j=0}^{\infty} (j+1)^{2/\delta}\alpha(j)\biggr)^{\delta/(2+\delta)}\left(\mE\bigl[\,|\eta_0(0)-\mu_0|^{2+\delta}\,\bigr]\right)^{2/(2+\delta)}<\infty.\end{aligned} \label{slim-10}\ee 

Next a lemma for the other series in \eqref{slim-9}.  
\begin{lemma}\label{a-lm3} For a finite constant $C$
\be\label{a-bd} 
\sup_{n\in\N, k\in\bZ} \biggl\lvert\sum_{i\in\bZ}a_{n,i}a_{n,i+k}\biggr\rvert \le C<\infty. 
 \ee
For all $k\in\bZ$,
\be \lim_{n\to\infty}\sum_{i\in\bZ}a_{n,i}a_{n,i+k} = \sum_{1\le j_1, j_2\le N}\theta_{j_1}\theta_{j_2}\Gamma_2\left((t_{j_1},r_{j_1}),(t_{j_2},r_{j_2})\right).\label{a-sum-limit}\ee
\end{lemma}

\begin{proof} By the  Schwarz inequality 
\[   \biggl\lvert\sum_{i\in\bZ}a_{n,i}a_{n,i+k}\biggr\rvert \le \sum_{i\in\bZ}a_{n,i}^2 \]
and so \eqref{a-bd} follows from the finite limit in \eqref{a-sum-limit} for $k=0$. 

To prove \eqref{a-sum-limit} expand the sum: 
\be\label{a-sum}\begin{aligned}
&\sum_{i\in\bZ}a_{n,i}a_{n,i+k}\\
&=n^{-1/2}\sum_{1\le j_1, j_2\le N}\theta_{j_1}\theta_{j_2}\sum_{i>0, i+k>0} P \bigl(X_{\lfloor nt_{j_1}\rfloor}^{\lfloor nt_{j_1}b\rfloor +\lfloor r_{j_1}\sqrt{n}\rfloor}\ge i\bigr) P \bigl(X_{\lfloor nt_{j_2}\rfloor}^{\lfloor nt_{j_2}b\rfloor +\lfloor r_{j_2}\sqrt{n}\rfloor}\ge i+k\bigr)
\\
&-n^{-1/2}\sum_{1\le j_1, j_2\le N}\theta_{j_1}\theta_{j_2}\sum_{i>0, i+k\le 0} P \bigl(X_{\lfloor nt_{j_1}\rfloor}^{\lfloor nt_{j_1}b\rfloor +\lfloor r_{j_1}\sqrt{n}\rfloor}\ge i\bigr) P \bigl(X_{\lfloor nt_{j_2}\rfloor}^{\lfloor nt_{j_2}b\rfloor +\lfloor r_{j_2}\sqrt{n}\rfloor}< i+k\bigr)
\\
&-n^{-1/2}\sum_{1\le j_1, j_2\le N}\theta_{j_1}\theta_{j_2}\sum_{i\le 0, i+k> 0} P \bigl(X_{\lfloor nt_{j_1}\rfloor}^{\lfloor nt_{j_1}b\rfloor +\lfloor r_{j_1}\sqrt{n}\rfloor}<i\bigr) P \bigl(X_{\lfloor nt_{j_2}\rfloor}^{\lfloor nt_{j_2}b\rfloor +\lfloor r_{j_2}\sqrt{n}\rfloor}\ge i+k\bigr)
\\
&+n^{-1/2}\sum_{1\le j_1, j_2\le N}\theta_{j_1}\theta_{j_2}\sum_{i\le 0, i+k\le 0} P \bigl(X_{\lfloor nt_{j_1}\rfloor}^{\lfloor nt_{j_1}b\rfloor +\lfloor r_{j_1}\sqrt{n}\rfloor}<i\bigr) P \bigl(X_{\lfloor nt_{j_2}\rfloor}^{\lfloor nt_{j_2}b\rfloor +\lfloor r_{j_2}\sqrt{n}\rfloor}< i+k\bigr).
\end{aligned}\ee
The middle two terms above   are $O(n^{-1/2})$ and hence  vanish as $n\to\infty$.  

The individual probabilities converge by the central limit theorem:  with $i=\fl{x\sqrt n}$, 
\begin{align*}
 P \bigl(X_{\fl{nt}}^{\fl{ntb} +\fl{r\sqrt{n}}}\ge i\bigr)
&=   P \Biggl(\frac{X_{\fl{nt}}^{0}+ \fl{ntb}}{\sqrt n} \ge \frac{\fl{x\sqrt n} -  \fl{r\sqrt{n}}}{\sqrt n}\Biggr)\\
&\longrightarrow   P (  B_{\sigma_1^2t} \ge x-r)=   P (  B_{\sigma_1^2t} \le r-x).  
\end{align*} 
The first  and the fourth sums in \eqref{a-sum}  are    handled by  Riemann sum arguments, with estimates to control the tails $\abs{i}\ge C\sqrt n$,   to obtain 
\begin{align*}
&\lim_{n\to\infty}n^{-1/2}\sum_{i>0, i+k>0} P \bigl(X_{\lfloor nt_{j_1}\rfloor}^{\lfloor nt_{j_1}b\rfloor +\lfloor r_{j_1}\sqrt{n}\rfloor}\ge i\bigr) P \bigl(X_{\lfloor nt_{j_2}\rfloor}^{\lfloor nt_{j_2}b\rfloor +\lfloor r_{j_2}\sqrt{n}\rfloor}\ge i+k\bigr)\\
&=\int_0^{+\infty} P (B_{\sigma_1^2t_{j_1}}\le r_{j_1}-x) P (B_{\sigma_1^2t_{j_2}}\le r_{j_2}-x)dx 
\end{align*}
and
\begin{align*}
&\lim_{n\to\infty}n^{-1/2}\sum_{i\le 0, i+k\le 0} P \bigl(X_{\lfloor nt_{j_1}\rfloor}^{\lfloor nt_{j_1}b\rfloor +\lfloor r_{j_1}\sqrt{n}\rfloor}<i\bigr) P \bigl(X_{\lfloor nt_{j_2}\rfloor}^{\lfloor nt_{j_2}b\rfloor +\lfloor r_{j_2}\sqrt{n}\rfloor}< i+k\bigr)\\
&=\int_{-\infty}^0 P (B_{\sigma_1^2t_{j_1}}>r_{j_1}-x) P (B_{\sigma_1^2t_{j_2}}>r_{j_1}-x)dx.
\end{align*}
More details can be found in \cite{zhai-phd}.  
By \eqref{2-r7}, this gives the conclusion.  
\end{proof}


We   let $n\to \infty$ in \eqref{slim-9}, with bounds \eqref{slim-10} and \eqref{a-bd}  and with limit \eqref{a-sum-limit} apply the dominated convergence theorem, and recall the definition \eqref{eta-9}  of   $\varrho_0$,  to conclude that
\be \lim_{n\to\infty} \bar{\sigma}_n^2=\varrho_0\sum_{1\le j_1, j_2\le N}\theta_{j_1}\theta_{j_2}\Gamma_2\left((t_{j_1},r_{j_1}),(t_{j_2},r_{j_2})\right).\label{var-limit}\ee

If the limit on the right of \eqref{var-limit} vanishes then   $\sum_{j=1}^N\theta_j \overline{S}_n(t_j,r_j)$ converges weakly  
to zero.   We continue by assuming that the  limit on the right of \eqref{var-limit} is strictly positive.  The proof of Proposition \ref{lmm3-8} for case (b) will be completed by the following central limit theorem for linear processes due to  Peligrad and Utev  \cite[Theorem 2.2(c)]{peli-utev-97}:

\begin{theorem}\label{LPCLT}
Let $\{b_{n,i}: -m_n\le i\le m_n, n\in \bZ_+\}$ be  real numbers   that  satisfy 
\be \limsup_{n\to\infty}\sum_{i\in\bZ}b_{n,i}^2<\infty\label{LPCLT-condition1}\ee
and 
\be \lim_{n\to\infty}\max_{i\in\bZ}\abs{b_{n,i}}=0, \label{LPCLT-condition2}\ee
 where $b_{n,i}=0$ for $|i|>m_n$.  
Let  $\{z(i): i\in\bZ\}$ be a centered, strongly mixing and non-degenerate $(\Var[z(0)]>0)$ stationary sequence with strong mixing coefficients $\{\alpha(j)\}$  such that
\be \Var \Bigl[\,\sum_{i=-m_n}^{m_n}b_{n,i}z(i)\Bigr]=1\label{LPCLT-condition3}\ee
and there exists $\delta>0$ so that $E|z(0)|^{2+\delta}<\infty$ and $\sum_{j=0}^\infty (j+1)^{2/\delta}\alpha(j)<\infty$.  
Then, 
\be \sum_{i=-m_n}^{m_n}b_{n,i}z(i)\Rightarrow \cN(0, 1) \quad \mbox{as}\hspace{2mm}n\to\infty.\ee
\end{theorem}

Apply this to  $b_{n,i}=a_{n,i}/\bar{\sigma}_n$ and $z(i)=\eta_0(i)-\mu_0$.
Since the random walk steps are bounded,  $a_{n,i}=0$ for large enough $i$ when $n$ is given.   
 From \eqref{a-sum-limit} and \eqref{var-limit},  
$$\lim_{n\to\infty} \sum_{i\in\bZ}b_{n,i}^2= \varrho_0^{-1}<\infty$$
and condition \eqref{LPCLT-condition1} is satisfied.  
 By  \eqref{mid-2}  
$\abs{a_{n,i}}\le n^{-1/4}\sum_{j=1}^N\abs{\theta_j}$ and thereby 
$$\max_{i\in\bZ}\abs{b_{n,i}}\le \frac{1}{n^{1/4}\bar{\sigma}_n}\sum_{j=1}^N\abs{\theta_j}\to 0 \quad \mbox{ as }n\to\infty, $$
verifying \eqref{LPCLT-condition2}.  
  Theorem \ref{LPCLT} gives 
  $ {\bar{\sigma}_n}^{-1}\sum_{i\in \bZ}a_{n,i}\bigl(\eta_0(i)-\mu_0\bigr)\Rightarrow \cN(0,1)$
which, 
combined   with \eqref{var-limit}, yields 
$$\sum_{j=1}^N\theta_i \overline{S}_n(t_j,r_j)\Rightarrow \cN\Bigl(0\, ,\varrho_0\!\!\sum_{1\le j_1, j_2\le N} \!\!\theta_{j_1}\theta_{j_2}\Gamma_2\left((t_{j_1},r_{j_1}),(t_{j_2},r_{j_2})\right)\Bigr).$$
This concludes the proof of Proposition \ref{lmm3-8} of case (b). 

\smallskip 

Case (c). It remains to consider the case where the initial increment sequence $\eta_0$ is defined by \eqref{(c)}.    
 Let $\ell(n)$ be any increasing sequence   such that $\lim_{n\to\infty}n/\sqrt{\ell(n)}=0$.   Split  \eqref{mid-2b}, with $\mu_0=0$ now,   into two terms: 
\be\label{S-split} \begin{aligned} 
&\sum_{j=1}^N\theta_j \overline{S}_n(t_j,r_j)=\sum_{i\in\bZ} a_{n,i}\eta_0(i) 
=\sum_{k=0}^{\ell(n)-1}\sum_{i,j \in\bZ} a_{n,i}\xi_{-k}(j)[ p^k(i,j)-p^k(i-1,j)] \\
& \qquad  + \sum_{k=\ell(n)}^{\infty}\sum_{i,j \in\bZ} a_{n,i}\xi_{-k}(j)[ p^k(i,j)-p^k(i-1,j)] =T_{n,1}+T_{n,2}. 
\end{aligned}\ee 
In the next lemma we show that the term $T_{n,2}$ is negligible.  
 
\begin{lemma}\label{lmm3-9case2} With $\ddd\lim_{n\to\infty}n/\sqrt{\ell(n)}=0$,  we have 
$\ddd\lim_{n\to\infty}  \E[\,\abs{T_{n,2}}^2\,]\to 0$.  
\end{lemma}
\begin{proof}  Due to the bounded range of the jump kernel $p(0,j)$,  after the first equality sign below  the sums   over 
$i_1$ and  $i_2$ have $O(n)$ terms, and  for a fixed $k$ the sum  over   $j$ is also  finite.  
\begin{align}
 \E[\,\abs{T_{n,2}}^2\,] 
&=\sigma_{\xi}^2\sum_{k=\ell(n)}^\infty\sum_{i_1, i_2\in\bZ} a_{n,i_1}a_{n,i_2}  \sum_{j\in\bZ}( p^k_{i_1,\,j}-p^k_{i_1-1,\,j})(p^k_{i_2,\,j}-p^k_{i_2-1,\,j}) \notag\\
&=\sigma_{\xi}^2\sum_{k=\ell(n)}^\infty\sum_{i_1, i_2\in\bZ}a_{n,i_1}a_{n,i_2}(2q^{k}_{i_2-i_1,0}-q^{k}_{i_2-i_1+1,0}-q^{k}_{i_2-i_1-1,0})\notag\\
&\le \sigma_{\xi}^2\sum_{j\in\bZ}\; \biggl\lvert \sum_{k=\ell(n)}^\infty  (2q^{k}_{j,0}-q^{k}_{j+1,0}-q^{k}_{j-1,0})\biggr\rvert \cdot \biggl\lvert\sum_{i\in\bZ}a_{n,i}a_{n,i+j}\biggr\rvert. \label{T-8}
\end{align}
Bound the middle sum using the characteristic function   $\phi_q(\theta)=\sum_{j\in\bZ}q(0,j)e^{\iota j\theta}$. 
\begin{align*}
&\biggl\lvert \sum_{k=\ell(n)}^\infty  (2q^{k}_{j,0}-q^{k}_{j+1,0}-q^{k}_{j-1,0})\biggr\rvert 
\\ &
= 
 \biggl\lvert \sum_{k=\ell(n)}^\infty \frac{1}{2\pi}\int_{-\pi}^\pi \phi_q^k(\theta)\bigl(2e^{-\iota j\theta}-e^{-\iota (j+1)\theta}-e^{-\iota (j-1)\theta}\bigr)d\theta \biggr\rvert\\
&= \biggl\lvert  \frac{1}{\pi} \int_{-\pi}^\pi \frac{\phi_q^{\ell(n)}(\theta)(1-\cos\theta)}{1-\phi_q(\theta)}e^{-\iota j\theta}d\theta  \biggr\rvert
\le C\int_{-\pi}^{\pi}\phi_q^{\ell(n)}(\theta)\,d\theta=Cq^{\ell(n)}(0,0)\\
&\le \frac{C}{\sqrt{\ell(n)}}. 
\end{align*} 
The ratio  $(1-\cos\theta)/(1-\phi_q(\theta))$  is bounded  over $[-\pi,\pi]$ because  it has  a finite limit   at $\theta=0$ and by the span 1 property $\phi_q(\theta)=1$ only at $\theta=0$.  



The sum over $j$ on line \eqref{T-8} has $O(n)$ nonzero terms.  Recalling \eqref{a-bd}, we conclude that  $ \E[\,\abs{T_{n,2}}^2\,] \le C{n}/{\sqrt{\ell(n)}}\to 0$ as $n\to\infty$. 
\end{proof}

Let  $U_{n,k}(j)=\xi_{-k}(j)\sum_{i\in\bZ}a_{n,i}(p^k_{i,j}-p^k_{i-1,j})$ so that   $T_{n,1}=\sum_{k=0}^{\ell(n)-1}\sum_{j \in\bZ}U_{n,k}(j)$ is a finite sum of independent mean zero random variables. 
To prove the goal \eqref{S-goal},   the Lindeberg-Feller theorem   shows that 
$T_{n,1} \Rightarrow \sum_{j=1}^N\theta_j S(t_j,r_j) $.  First the limiting variance. 
\be\label{T-15}\begin{aligned}
&\E[T_{n,1}^2] =\sigma_{\xi}^2\sum_{j\in\bZ}\sum_{k=0}^{\ell(n)-1}  \bigl[2q^{k}(j,0)-q^{k}(j+1,0)-q^{k}(j-1,0)\bigr] \sum_{i\in\bZ}a_{n,i}a_{n,i+j}\\
&=\sigma_{\xi}^2\sum_{j\in\bZ}\left[a(j-1)+a(j+1)-2a(j)\right]\sum_{i\in\bZ}a_{n,i}a_{n,i+j}
\; - \; \E[T_{n,2}^2] \\
&\underset{n\to\infty}\longrightarrow  \ \ 
 \frac{\sigma_{\xi}^2}{\sigma_1^2}\sum_{1\le j_1, \, j_2\le N}\theta_{j_1}\theta_{j_2}\Gamma_2\left((t_{j_1},r_{j_1}),(t_{j_2},r_{j_2})\right).
\end{aligned}\ee
The last limit   came from a combination of \eqref{cov-sum} 
and Lemmas \ref{a-lm3} and \ref{lmm3-9case2} .

For  the Lindeberg condition (condition (ii) of Theorem 3.4.5 on p.~129 in \cite{durr})  
the task is to show that, for all $\e>0$, 
\be\label{LF19}  \sum_{j\in\bZ}\sum_{k=0}^{\ell(n)-1}\bE\bigl[U^2_{n,k}(j)\ind\{\left|U_{n,k}(j)\right|\ge\e\}\bigr] \to 0 . 
\ee
By Schwarz and Markov inequalities, 
\begin{align}
&\sum_{j\in\bZ}\sum_{k=0}^{\ell(n)-1}\bE\bigl[U^2_{n,k}(j)\ind\{\left|U_{n,k}(j)\right|\ge\e\}\bigr]
\ \le \  \frac{1}{\e^2}\sum_{j\in\bZ}\sum_{k=0}^{\ell(n)-1}\bE\bigl[U^4_{n,k}(j)\bigr] \nn\\
&=\frac{\bE\bigl[\xi^4_{0}(0)\bigr]}{\e^2}\sum_{j\in\bZ}\sum_{k=0}^{\ell(n)-1}\Bigl(\,\sum_{i\in\bZ}a_{n,i} [ p^k(i,j)-p^k(i-1,j)]\Bigr)^4. \label{line-T6}
\end{align}
To two powers of the quantity in parentheses apply the inequality  
$$\sum_{i\in\bZ}\,\abs{a_{n,i}}\,[ p^k(i,j)+p^k(i-1,j) ]  \le Cn^{-1/4}.$$ 
Then line \eqref{line-T6} is bounded above by 
\begin{align}
   \frac{C}{\e^2n^{1/2}}\sum_{j\in\bZ}\sum_{k=0}^{\ell(n)-1}\Bigl(\,\sum_{i\in\bZ}a_{n,i} [ p^k(i,j)-p^k(i-1,j)]\Bigr)^2  \label{case2-eq5}
\end{align}
which vanishes as $n\to\infty$ because the finite limit of the double sum of squares is exactly what was treated in \eqref{T-15}.   This verifies \eqref{LF19}.  
Collecting the pieces, we have  showed that  $T_{n,1} \Rightarrow \sum_{j=1}^N\theta_j S(t_j,r_j) $.  This,  together with  \eqref{S-split} and Lemma \ref{lmm3-9case2},   implies the goal \eqref{S-goal}. Proposition \ref{lmm3-8} has now been proved also for case (c) of Theorem  \ref{thm2-4}.  
\end{proof}


To complete the proof of Theorem \ref{thm2-4}, apply Propositions \ref{lmm3-10} and  \ref{lmm3-8} to the right-hand side of decomposition \eqref{2-7}, and use  the independence of  the processes $\overline{S}_n$ and $\overline{F}_n$.

\section{Height fluctuations: process-level convergence}\label{sec:tight}
 
In this section we prove Theorem \ref{thm2-6}.   To simplify the exposition we take  $Q=[0,1]^2$.
Theorem  2 in 
  Bickel and Wichura  \cite{bick-wich} gives the following  necessary and sufficient condition   for  weak convergence $X_n\Rightarrow X$ of  $D_2$-valued processes. 
\begin{enumerate}
\item[(i)] (Convergence of finite-dimensional distributions.) For all finite sets $\{(t_i,r_i)\}_{i=1}^N\subset [0,1]^2$, we have
$$\bigl(X_n(t_1,r_1),\ldots,X_n(t_N,r_N)\bigr)\Rightarrow \bigl(X(t_1,r_1),\ldots,X(t_N,r_N)\bigr).$$
\item[(ii)] (Tightness.)   $\forall \e>0$, $\lim_{\delta\to 0}\limsup_{n\to\infty} P \{w'_\delta(X_n)\ge \e\}=0$, for the modulus  
$$w'_\delta(x)=\inf_\Delta\max_{G\in\Delta}\sup_{(t,r),(s,q)\in G}|x(t,r)-x(s,q)|, \quad x\in D_2, $$ where the infimum is over partitions  $\Delta$  of $[0,1]^2$ formed by finitely many lines parallel to the coordinate axes and such that any element $G$ of $\Delta$ is a left-closed, right-open rectangle with diameter at least $\delta$. 
\end{enumerate} 

We   proved the finite-dimensional marginal convergence (i)  in Theorem \ref{thm2-4}. For the tightness proof  that follows we check   the sufficient conditions given by the next  lemma    (Proposition 2 in \cite{kuma-08}) in terms of  the modulus of continuity  
\be w_\delta(x)=\sup_{\begin{subarray}{l} (t,r),(s,q)\in [0,1]^2\\ \abs{(t,r)-(s,q)}<\delta\end{subarray}}|x(t,r)-x(s,q)| , \quad x\in D_2.  \label{tight-1}\ee

\begin{lemma}\label{lmm3-11}
Let  $\{X_n\}$ be a sequence of $D_2$-valued processes.  Assume  that  there exists a decreasing sequence $\delta_n\searrow 0$ such that
\begin{enumerate}
\item[(i)]  $\exists$  $\beta>0$, $\kappa>2$, and $C>0$ such that,  for all large enough $n$,
\be E(|X_n(t,r)-X_n(s,q)|^\beta)\le C\abs{(t,r)-(s,q)}^\kappa\label{tight-2}\ee
holds for all $(t,r),(s,q)\in [0,1]^2$ at  distance $\abs{(t,r)-(s,q)}> \delta_n$;\\
\item[(ii)]    $\forall\e,\eta>0$, there exists an $n_0>0$ such that  for all $n\ge n_0$,
\be P\{w_{\delta_n}(X_n)>\e\}< \eta.\label{tight-3}\ee 
\end{enumerate}
Then, for all $\e,\eta>0$, there exist $0<\delta<1$ and integer $n_0<\infty$ such that  
$$P\{w_{\delta}(X_n)\ge\e\}\le\eta\quad \forall n\ge n_0.$$
\end{lemma}

The two tightness conditions (i) and (ii)  are checked in Lemmas \ref{lmm3-12} and \ref{lmm3-14}.
 
\begin{lemma}\label{lmm3-12}  Assume the assumptions of Theorem \ref{thm2-6}. 
Fix  $\kappa$ and $\gamma$ such that  $2<\kappa<3$ and $0<\gamma\le \frac{3}{\kappa}$ and let $\delta_n=n^{-\gamma}$.
Then, there exists a constant $C>0$ such that  for all sufficiently large $n$, 
\be\label{Y-98}  \mE\bigl(|\fluc_n(t,r)-\fluc_n(s,q)|^{12}\bigr)\le C\abs{(t,r)-(s,q)}^\kappa\ee
for all $t,s,r,q\in [0,1]$ with $\abs{(t,r)-(s,q)}>n^{-\gamma}$.  
\end{lemma}
\begin{proof}

From decomposition \eqref{2-7},  for a constant $C$, 
\be \begin{aligned}\label{tight-4}
& C^{-1}  \mE\left(|\fluc_n(t,r)-\fluc_n(s,q)|^{12}\right)  \; \le \;    |\overline{H}_n(t,r)-\overline{H}_n(s,q)|^{12}   \\
& \qquad
+ \bE\left(|\overline{F}_n(t,r)-\overline{F}_n(s,q)|^{12}\right)
 + \mE\left(|\overline{S}_n(t,r)-\overline{S}_n(s,q)|^{12}\right) .
\end{aligned}\ee
Estimate \eqref{Y-98}  comes by treating each term on the right-hand side of \eqref{tight-4} in turn. 
 
 \smallskip 
 
As observed in the early part of Section \ref{sec:fd}, $\overline{H}_n(t,r)=O(n^{-1/4})$ uniformly, and  the   first term on the right of \eqref{tight-4} satisfies 
\be\label{Y-101} 
 |\overline{H}_n(t,r)-\overline{H}_n(s,q)|^{12}  \le Cn^{-3}\le Cn^{-\gamma\kappa}  \le C\abs{(t,r)-(s,q)}^\kappa . 
\ee



The following argument will be used more than once:  if $\{\zeta_i\}$ are i.i.d.\ mean zero variables with a finite 12th moment and $\abs{a_i}\vee 1\le B$, then 
\be\begin{aligned}
&E\biggl[\biggl(\sum_i a_i\zeta_i\biggr)^{12}\,\biggr] 
= E\sum_{i_1,i_2,\dotsc, i_{12}} a_{i_1}\dotsm a_{i_{12}}  \, \zeta_{i_1}\dotsm \zeta_{i_{12}} \\
&\le  C \sum_{\substack{1\le k\le 6\\ m_i\ge 2:\,m_1+\dotsm+m_k=12}}   \Bigl( \,\sum_{i_1} \abs{a_{i_1}}^{m_1}\Bigr) \Bigl( \,\sum_{i_2} \abs{a_{i_2}}^{m_2}\Bigr)  \dotsm \Bigl( \,\sum_{i_k} \abs{a_{i_k}}^{m_k}\Bigr)  \\
&\le B^{10} C\Bigl(1+ \sum_i a_i^2\Bigr)^6   
\end{aligned}\label{t:help7}\ee
where $C$ is a combinatorial constant.  


  For the second term on the right of \eqref{tight-4}, recalling \eqref{2-6}, and with $t\ge s$, 
\begin{align*}
&\overline{F}_n(t,r)-\overline{F}_n(s,q)\\
&=n^{-1/4}\sum_{k=1}^{\lfloor ns\rfloor}\sum_{x\in\mathbb{Z}}\xi_k(x)\left[ P \bigl(X_{\lfloor nt\rfloor -k}^{\lfloor ntb\rfloor +\lfloor r\sqrt{n}\rfloor}=x\bigr)- P \bigl(X_{\lfloor ns\rfloor -k}^{\lfloor nsb\rfloor +\lfloor q\sqrt{n}\rfloor}=x\bigr)\right]\\
&\qquad +n^{-1/4}\sum_{k=\lfloor ns\rfloor+1}^{\lfloor nt\rfloor}\sum_{x\in\mathbb{Z}}\xi_k(x) P \bigl(X_{\lfloor nt\rfloor -k}^{\lfloor ntb\rfloor +\lfloor r\sqrt{n}\rfloor}=x\bigr).
\end{align*}
Apply \eqref{t:help7}:  
\begin{align*}
&\bE|\overline{F}_n(t,r)-\overline{F}_n(s,q)|^{12}\\
&\le  Cn^{-3}\biggl\{1+\sum_{k=1}^{\lfloor ns\rfloor}\sum_{x\in\mathbb{Z}}\left[ P \bigl(X_{\lfloor nt\rfloor -k}^{\lfloor ntb\rfloor +\lfloor r\sqrt{n}\rfloor}=x\bigr)- P \bigl(X_{\lfloor ns\rfloor -k}^{\lfloor nsb\rfloor +\lfloor q\sqrt{n}\rfloor}=x\bigr)\right]^2
\\
&\qquad 
+\sum_{k=\lfloor ns\rfloor+1}^{\lfloor nt\rfloor}\sum_{x\in\mathbb{Z}} P \bigl(X_{\lfloor nt\rfloor -k}^{\lfloor ntb\rfloor +\lfloor r\sqrt{n}\rfloor}=x\bigr)^2\biggr\}^6. 
\end{align*}
The sums inside the braces develop as follows.  
\be\label{tight-12}\begin{aligned}
&\sum_{k=1}^{\lfloor ns\rfloor}\sum_{x\in\mathbb{Z}}\left[ P \bigl(X_{\lfloor nt\rfloor -k}^{\lfloor ntb\rfloor +\lfloor r\sqrt{n}\rfloor}=x\bigr)- P \bigl(X_{\lfloor ns\rfloor -k}^{\lfloor nsb\rfloor +\lfloor q\sqrt{n}\rfloor}=x\bigr)\right]^2\\
&\qquad\qquad\qquad\qquad 
+\sum_{k=\lfloor ns\rfloor+1}^{\lfloor nt\rfloor}\sum_{x\in\mathbb{Z}} P \bigl(X_{\lfloor nt\rfloor -k}^{\lfloor ntb\rfloor +\lfloor r\sqrt{n}\rfloor}=x\bigr)^2
\\
&=\sum_{k=1}^{\lfloor nt\rfloor} P \bigl(X_{\lfloor nt\rfloor -k}^{\lfloor ntb\rfloor +\lfloor r\sqrt{n}\rfloor}-\tilde{X}_{\lfloor nt\rfloor -k}^{\lfloor ntb\rfloor +\lfloor r\sqrt{n}\rfloor}=0\bigr)\\
&\qquad\qquad\qquad\qquad 
+\sum_{k=1}^{\lfloor ns\rfloor} P \bigl(X_{\lfloor ns\rfloor -k}^{\lfloor nsb\rfloor +\lfloor q\sqrt{n}\rfloor}-\tilde{X}_{\lfloor ns\rfloor -k}^{\lfloor nsb\rfloor +\lfloor q\sqrt{n}\rfloor}=0\bigr)
\\
&-2\sum_{k=1}^{\lfloor ns\rfloor} P \bigl(X_{\lfloor nt\rfloor -k}^{\lfloor ntb\rfloor +\lfloor r\sqrt{n}\rfloor}-\tilde{X}_{\lfloor ns\rfloor -k}^{\lfloor nsb\rfloor +\lfloor q\sqrt{n}\rfloor}=0\bigr),
\end{aligned}\ee
where $X_\centerdot$ and $\tilde{X}_\centerdot$ are independent random walks with   transition probability $p$.  

Recalling   transition probability $q$ from \eqref{21-d2}, 
$$\sum_{k=1}^{\lfloor nt\rfloor} P \bigl(X_{\lfloor nt\rfloor -k}^{\lfloor ntb\rfloor +\lfloor r\sqrt{n}\rfloor}-\tilde{X}_{\lfloor nt\rfloor -k}^{\lfloor ntb\rfloor +\lfloor r\sqrt{n}\rfloor}=0\bigr)=\sum_{k=1}^{\lfloor nt\rfloor}q^{\lfloor nt\rfloor -k}(0,0)=\sum_{k=0}^{\lfloor nt\rfloor-1}q^k(0,0).$$
Similarly,
$$\sum_{k=1}^{\lfloor ns\rfloor} P \bigl(X_{\lfloor ns\rfloor -k}^{\lfloor nsb\rfloor +\lfloor q\sqrt{n}\rfloor}-\tilde{X}_{\lfloor ns\rfloor -k}^{\lfloor nsb\rfloor +\lfloor q\sqrt{n}\rfloor}=0\bigr)=\sum_{k=0}^{\lfloor ns\rfloor -1}q^k(0,0).$$
And,
\begin{align*}
&\sum_{k=1}^{\lfloor ns\rfloor} P \bigl(X_{\lfloor nt\rfloor -k}^{\lfloor ntb\rfloor +\lfloor r\sqrt{n}\rfloor}-\tilde{X}_{\lfloor ns\rfloor -k}^{\lfloor nsb\rfloor +\lfloor q\sqrt{n}\rfloor}=0\bigr)=\sum_{k=0}^{\lfloor ns\rfloor-1} P \bigl(X_{\lfloor nt\rfloor -\lfloor ns\rfloor+k}^{\lfloor ntb\rfloor +\lfloor r\sqrt{n}\rfloor}-\tilde{X}_{k}^{\lfloor nsb\rfloor +\lfloor q\sqrt{n}\rfloor}=0\bigr)\\
&\qquad=\sum_{k=0}^{\lfloor ns\rfloor-1} E\Bigl[  P \bigl(X_k^{X_{\lfloor nt\rfloor -\lfloor ns\rfloor}^{\lfloor ntb\rfloor +\lfloor r\sqrt{n}\rfloor}}-\tilde{X}_{k}^{\lfloor nsb\rfloor +\lfloor q\sqrt{n}\rfloor}=0\bigr)\Bigr]\\
&\qquad =E\biggl[\; \sum_{k=0}^{\lfloor ns\rfloor-1} q^k\bigl(X_{\lfloor nt\rfloor -\lfloor ns\rfloor}^{\lfloor ntb\rfloor +\lfloor r\sqrt{n}\rfloor}-\lfloor nsb\rfloor -\lfloor q\sqrt{n}\rfloor,0\bigr)\biggr] .
\end{align*}
In sum, we   rewrite  the right-hand side of \eqref{tight-12} as
\begin{align}
&\sum_{k=0}^{\lfloor nt\rfloor-1}q^k(0,0)+\sum_{k=0}^{\lfloor ns\rfloor -1}q^k(0,0)-2E\left\{\sum_{k=0}^{\lfloor ns\rfloor-1} q^k\bigl(X_{\lfloor nt\rfloor -\lfloor ns\rfloor}^{\lfloor ntb\rfloor +\lfloor r\sqrt{n}\rfloor}-\lfloor nsb\rfloor -\lfloor q\sqrt{n}\rfloor,0\bigr)\right\}\notag\\
&=\sum_{k=\lfloor ns\rfloor}^{\lfloor nt\rfloor-1}q^k(0,0)+2E\left\{\sum_{k=0}^{\lfloor ns\rfloor-1}\left[q^k(0,0)-q^k\bigl(X_{\lfloor nt\rfloor -\lfloor ns\rfloor}^{\lfloor ntb\rfloor +\lfloor r\sqrt{n}\rfloor}-\lfloor nsb\rfloor -\lfloor q\sqrt{n}\rfloor,0\bigr)\right]\right\}.\label{tight-13}
\end{align}
The  first term in \eqref{tight-13} is  bounded by
\be\sum_{k=\lfloor ns\rfloor}^{\lfloor nt\rfloor-1}q^k(0,0)\le \sum_{k=\lfloor ns\rfloor}^{\lfloor nt\rfloor-1}\frac{C}{\sqrt{k}}\le\sum_{k=1}^{\lfloor nt\rfloor-\lfloor ns\rfloor}\frac{C}{\sqrt{k}}\le C\left[1+\sqrt{(t-s)n}\right].\label{tight-14}\ee
To bound the  second term in \eqref{tight-13}, observe first that the terms in the potential kernel   \eqref{t37-2-1}  are nonnegative: 
\begin{align*}  q^k(x,0)&=\sum_{y\in\bZ}p^k(x,y)p^k(0,y)\le \tfrac{1}{2}\sum_{y\in\bZ}\bigl[\bigl(p^k(x,y)\bigr)^2+\bigl(p^k(0,y)\bigr)^2\bigr]=\sum_{y\in\bZ}\bigl(p^k(0,y)\bigr)^2\\
&=q^k(0,0).\end{align*} 
Now  the second term in \eqref{tight-13} is bounded by  
\begin{align}
&E\biggl\{\;\sum_{k=0}^{\lfloor ns\rfloor-1}\left[q^k(0,0)-q^k\bigl(X_{\lfloor nt\rfloor -\lfloor ns\rfloor}^{\lfloor ntb\rfloor +\lfloor r\sqrt{n}\rfloor}-\lfloor nsb\rfloor -\lfloor q\sqrt{n}\rfloor,0\bigr)\right]\biggr\}\notag\\
&\le E \Bigl[ a\bigl(X_{\lfloor nt\rfloor -\lfloor ns\rfloor}^{\lfloor ntb\rfloor +\lfloor r\sqrt{n}\rfloor}-\lfloor nsb\rfloor -\lfloor q\sqrt{n}\rfloor\bigr) \Bigr]\nn\\
&\le CE\bigl|X_{\lfloor nt\rfloor -\lfloor ns\rfloor}^{\lfloor ntb\rfloor +\lfloor r\sqrt{n}\rfloor}-\lfloor nsb\rfloor -\lfloor q\sqrt{n}\rfloor\bigr|\notag\\
&\le C\bigl({|t-s|^{1/2}}+|r-q|\bigr)\sqrt{n}+C,\label{tight-15}
\end{align}
where the second inequality is due to  
\be \lim_{x\to \pm\infty} {\abs x}^{-1}{a(x)}=\tfrac{1}{2} {\sigma_1^{-2}}, \label{nl38}\ee 
from  P28.4 on p.~345  of \cite{spitzer} and by the symmetry of $a(x)$.  

Combining \eqref{tight-14} and \eqref{tight-15} gives  a bound for \eqref{tight-12}, and hence the second term on the right of \eqref{tight-4} satisfies 
\be  \bE\left(|\overline{F}_n(t,r)-\overline{F}_n(s,q)|^{12}\right) \le C n^{-3}\left[\left(\sqrt{|t-s|}+|r-q|\right)\sqrt{n}+1\right]^{6}.\label{tight-16}\ee

\smallskip

  For the third term on the right of \eqref{tight-4}, from \eqref{2-5},  
\begin{align*}
&\overline{S}_n(t,r)-\overline{S}_n(s,q)\notag\\
&\quad 
=n^{-1/4}\sum_{i>0}\bigl(\eta_0(i)-\mu_0\bigr)\left[ P \bigl(i\le X_{\lfloor nt\rfloor}^{\lfloor ntb\rfloor +\lfloor r\sqrt{n}\rfloor}\bigr)- P \bigl(i\le X_{\lfloor ns\rfloor}^{\lfloor nsb\rfloor +\lfloor q\sqrt{n}\rfloor}\bigr)\right]\notag\\
&\qquad 
-\; n^{-1/4}\sum_{i\le 0}\bigl(\eta_0(i)-\mu_0\bigr)\left[ P \bigl(i>X_{\lfloor nt\rfloor}^{\lfloor ntb\rfloor +\lfloor r\sqrt{n}\rfloor}\bigr)- P \bigl(i>X_{\lfloor ns\rfloor}^{\lfloor nsb\rfloor +\lfloor q\sqrt{n}\rfloor}\bigr)\right]. 
\end{align*}
Note that $X^k_t\overset{d} =k+X^0_t$ and define    events
\begin{align*}
A_{1,i}&=\left\{X_{\lfloor nt\rfloor}^i\ge -\lfloor ntb\rfloor -\lfloor r\sqrt{n}\rfloor, X_{\lfloor ns\rfloor}^i< -\lfloor nsb\rfloor -\lfloor q\sqrt{n}\rfloor\right\},\\
A_{2,i}&=\left\{X_{\lfloor nt\rfloor}^i< -\lfloor ntb\rfloor -\lfloor r\sqrt{n}\rfloor, X_{\lfloor ns\rfloor}^i\ge -\lfloor nsb\rfloor -\lfloor q\sqrt{n}\rfloor\right\}.
\end{align*}
Then we can rewrite the difference above  as
\be\overline{S}_n(t,r)-\overline{S}_n(s,q)=n^{-1/4}\sum_{i\in \bZ}\bigl(\eta_0(-i)-\mu_0\bigr)\left[ P \bigl(A_{1,i}\bigr)- P \bigl(A_{2,i}\bigr)\right].\label{tight-6}\ee 
We prove an intermediate bound where the different cases (a), (b) and (c) of Theorem \ref{thm2-6} are felt.

\begin{lemma}\label{lm-S-99}   Suppose the initial increment sequence $\{\eta_0(x)\}$ satisfies one of the   assumptions  {\rm (a), (b)} and {\rm(c)} of Theorem \ref{thm2-6}.  Then  we have the inequality 
\be\label{S-99} 
\mE \bigl[\,\bigl(\overline{S}_n(t,r)-\overline{S}_n(s,q)\bigr)^{12}\,\bigr]\le  Cn^{-3}\Bigl\{1+\sum_{m\in\bZ}\left[ P \bigl(A_{1,m}\bigr)+ P \bigl(A_{2,m}\bigr)\right]\Bigr\}^6. 
\ee
\end{lemma} 

\begin{proof}   
Case (a).  The   case of  i.i.d.\ $\{\eta_0(x)\}$ is handled by the argument in \eqref{t:help7}, 
beginning with \eqref{tight-6}. Details  
can be found in \cite{zhai-phd}. 

 \smallskip 
 
 Case (b).  For the case of strongly mixing initial increments we  state a  lemma   from Rio's lectures  (see Theorem 2.2 and the derivation of equation (C.6) in \cite{rio-lect}).
\begin{lemma}\label{lmm3-13} {\rm(a)}  
Let $m\in\bN$, $\{X_i\}_{i\in\bN}$  centered real  random variables with   strong mixing coefficients $\{\alpha(k)\}_{k\ge 0}$ and define 
\[ \alpha^{-1}(u)=\inf\{k\in\bZ_+: \alpha(k)\le u\} =\sum_{i\ge 0}\ind_{\{u< \alpha(i)\}}.\]    Assume $E\abs{X_i}^{2m}<\infty$ and let   $Q_k(u)$ be  the quantile function of $|X_k|$. 
 Let   $S_n=\sum_{k=1}^nX_k$. Then there exist  positive constants $a_m$ and $b_m$ such that 
\be\label{rio-10}\begin{aligned}
E \left(S_n^{2m}\right)&\le  a_m\left(\int_0^1\sum_{k=1}^n[\alpha^{-1}(u)\wedge n]Q_k^2(u)du\right)^m \\
&\qquad +  b_m\sum_{k=1}^n\int_0^1[\alpha^{-1}(u)\wedge n]^{2m-1}Q_k^{2m}(u)du. 
\end{aligned}\ee 

{\rm(b)} 
  Suppose centered  $\{X_i\}_{i\in\bN}$ have finite $r$th moment. Then for $p\in[1,r)$  there exists a constant $c_p>0$ such that
\be \begin{aligned}
&\sum_{k=1}^n\int_0^1[\alpha^{-1}(u)\wedge n]^{p-1}Q_k^{p}(u)du  \\
&\qquad \le \  c_p \left(\sum_{i=0}^n(i+1)^{(pr-2r+p)/(r-p)}\alpha(i)\right)^{1-p/r}\sum_{k=1}^n\left(E|X_k|^r\right)^{p/r}.\label{tight-6c}
\end{aligned}\ee
\end{lemma}

Apply this lemma to  representation \eqref{tight-6}.   Note that 
  $k_n=\#\{i\in\bZ:  P \bigl(A_{1,i}\bigr)- P \bigl(A_{2,i}\bigr)\neq 0\} =O(n)$ due to the bounded support of $p$ \eqref{w-ass}. Letting $Q_i$ be the quantile function of $\bigl|\eta_0(-i)-\mu_0\bigr|\cdot\left| P \bigl(A_{1,i}\bigr)- P \bigl(A_{2,i}\bigr)\right|$ and $m=6$, \eqref{rio-10} gives 
\begin{align*} 
&\mE \bigl[\left(\overline{S}_n(t,r)-\overline{S}_n(s,q)\right)^{12}\bigr]\\
&\le  Cn^{-3}\left[\left(\sum_{i\in\bZ}\int_0^1[\alpha^{-1}(u)\wedge k_n]Q_i^2(u)du\right)^6+\sum_{j\in\bZ}\int_0^1[\alpha^{-1}(u)\wedge k_n]^{11}Q_j^{12}(u)du\right].
\end{align*}

Let $p=2, r=12$ in \eqref{tight-6c}, we can get an upper bound for $\sum_{i\in\bZ}\int_0^1[\alpha^{-1}(u)\wedge k_n]Q_i^2(u)du$,
\begin{align*}
&\sum_{i\in\bZ}\int_0^1[\alpha^{-1}(u)\wedge k_n]Q_i^2(u)du\\
&\le  c_{2} \left(\sum_{i=0}^{k_n}(i+1)^{1/5}\alpha(i)\right)^{5/6}\sum_{j\in\bZ}\left(\mE|\eta_0(-j)-\mu_0\bigr|^{12}\right)^{1/6}\left| P \bigl(A_{1,j}\bigr)- P \bigl(A_{2,j}\bigr)\right|^2\\
&\le  C \left(\sum_{i=0}^{k_n}(i+1)^{1/5}\alpha(i)\right)^{5/6}\sum_{j\in\bZ}\left[ P \bigl(A_{1,j}\bigr)+ P \bigl(A_{2,j}\bigr)\right]\le C \sum_{j\in\bZ}\left[ P \bigl(A_{1,j}\bigr)+ P \bigl(A_{2,j}\bigr)\right].
\end{align*}
By the same token, let $p=12$, $r=12+\delta$ in \eqref{tight-6c}, we can show that
\begin{align*}
&\sum_{j\in\bZ}\int_0^1[\alpha^{-1}(u)\wedge k_n]^{11}Q_j^{12}(u)du\\
&\le  C \left(\sum_{i=0}^{k_n}(i+1)^{10+132/\delta}\alpha(i)\right)^{\delta/(12+\delta)}\sum_{j\in\bZ}\left[ P \bigl(A_{1,j}\bigr)+ P \bigl(A_{2,j}\bigr)\right]\\
&\le C \sum_{j\in\bZ}\left[ P \bigl(A_{1,j}\bigr)+ P \bigl(A_{2,j}\bigr)\right].
\end{align*}  
Hence, 
\be\label{tight-6b}\begin{aligned}
&\mE \bigl[\left(\overline{S}_n(t,r)-\overline{S}_n(s,q)\right)^{12}\bigr]\\
&\le   Cn^{-3}\biggl\{\Bigl(\sum_{j\in\bZ}\left[ P (A_{1,j})+ P (A_{2,j})\right]\Bigr)^6+\sum_{j\in\bZ}\left[ P (A_{1,j})+ P (A_{2,j})\right]\biggr\}\\
&\le  Cn^{-3}\Bigl\{1+\sum_{m\in\bZ}\left[ P (A_{1,m})+ P (A_{2,m})\right]\Bigr\}^6. 
\end{aligned}\ee


Case (c).   From \eqref{tight-6} and \eqref{(c)},  
$$\overline{S}_n(t,r)-\overline{S}_n(s,q)=n^{-1/4}\sum_{i,\,j\in\bZ}\sum_{k=0}^{\infty}\xi_{-k}(j) (p^k_{0,\,j+i}-p^k_{0,\,j+i+1})\left[ P (A_{1,i})- P (A_{2,i})\right].$$ 
Then by \eqref{t:help7} 
\be\label{S-106}\begin{aligned}
&\mE \bigl[\left(\overline{S}_n(t,r)-\overline{S}_n(s,q)\right)^{12}\bigr]\\
&\le  C n^{-3} \biggl(  1+  \sum_{j\in\bZ}\sum_{k=0}^{\infty}\, 
\biggl\lvert \sum_{i\in\bZ}\  (p^k_{0,\,j+i}-p^k_{0,\,j+i+1})\bigl[ P (A_{1,i})- P (A_{2,i})\bigr]\,\biggr\rvert^2 \; 
\biggr)^6 
\end{aligned}\ee 
Expand the sum of squares inside the parentheses: 
\begin{align}
 &\sum_{j\in\bZ}\sum_{k=0}^{\infty}\sum_{i_1,i_2\in \bZ}(p^k_{0,\,j+i_1}-p^k_{0,\,j+i_1+1})(p^k_{0,\,j+i_2}-p^k_{0,\,j+i_2+1})\notag\\
&\qquad \qquad\qquad
\times \left[ P \bigl(A_{1,i_1}\bigr)- P \bigl(A_{2,i_1}\bigr)\right]\left[ P \bigl(A_{1,i_2}\bigr)- P \bigl(A_{2,i_2}\bigr)\right]\notag\\
=&\sum_{i,\ell\in\bZ}\left[a(\ell-1)+a(\ell+1)-2a(\ell)\right] \cdot \left[ P \bigl(A_{1,i}\bigr)- P \bigl(A_{2,i}\bigr)\right]\left[ P \bigl(A_{1,i+\ell}\bigr)- P \bigl(A_{2,i+\ell}\bigr)\right]\notag\\
\le & \sum_{\ell\in\bZ}\left[a(\ell-1)+a(\ell+1)-2a(\ell)\right] \notag\\
&\qquad\qquad\qquad \times \; \sum_{i\in\bZ}\frac{1}{2}\left\{\left[ P \bigl(A_{1,i}\bigr)- P \bigl(A_{2,i}\bigr)\right]^2+\left[ P \bigl(A_{1,i+\ell}\bigr)- P \bigl(A_{2,i+\ell}\bigr)\right]^2\right\}\notag\\
= & \frac{1}{\sigma_1^2}\sum_{i\in\bZ}\left[ P \bigl(A_{1,i}\bigr)- P \bigl(A_{2,i}\bigr)\right]^2
\le \frac{1}{\sigma_1^2}\sum_{i\in\bZ}\left[ P \bigl(A_{1,i}\bigr)+ P \bigl(A_{2,i}\bigr)\right].\notag
\end{align}
where  the last equality is from \eqref{cov-sum}. 
 Inserting this into \eqref{S-106} gives  
\begin{align*}
&\mE \bigl[\left(\overline{S}_n(t,r)-\overline{S}_n(s,q)\right)^{12}\bigr]\le Cn^{-3}\Bigl\{1+\sum_{i\in\bZ}\left[ P (A_{1,i})+ P (A_{2,i})\right]\Bigr\}^6.
\end{align*}
The proof of Lemma \ref{lm-S-99} is complete. 
\end{proof} 

 We continue by bounding  the sum on the right of \eqref{S-99}.  Suppose $t\ge s$. 
\begin{align}
\sum_{m\in\bZ} P \bigl(A_{1,m}\bigr)&=\sum_{m\in\bZ} P \bigl(X_{\lfloor nt\rfloor}^0\ge -\lfloor ntb\rfloor -\lfloor r\sqrt{n}\rfloor-m, X_{\lfloor ns\rfloor}^0< -\lfloor nsb\rfloor -\lfloor q\sqrt{n}\rfloor-m\bigr)\notag\\
&=\sum_{m\in\bZ}\sum_{\ell>m} P \bigl(X_{\lfloor ns\rfloor}^0=-\lfloor nsb\rfloor -\lfloor q\sqrt{n}\rfloor-\ell\bigr)\notag\\
\times& \  P \bigl(X_{\lfloor nt\rfloor-\lfloor ns\rfloor}^0\ge \lfloor nsb\rfloor-\lfloor ntb\rfloor +\lfloor q\sqrt{n}\rfloor-\lfloor r\sqrt{n}\rfloor-m+\ell\bigr)\notag\\
\overset{k=\ell-m}{=}&\sum_{m\in\bZ}\sum_{k>0} P \bigl(X_{\lfloor ns\rfloor}^0=-\lfloor nsb\rfloor -\lfloor q\sqrt{n}\rfloor-k-m\bigr)\notag\\
\times&\  P \bigl(X_{\lfloor nt\rfloor-\lfloor ns\rfloor}^0\ge \lfloor nsb\rfloor-\lfloor ntb\rfloor +\lfloor q\sqrt{n}\rfloor-\lfloor r\sqrt{n}\rfloor+k\bigr)\notag\\
&=\sum_{k>0} P \bigl(X_{\lfloor nt\rfloor-\lfloor ns\rfloor}^0\ge \lfloor nsb\rfloor-\lfloor ntb\rfloor +\lfloor q\sqrt{n}\rfloor-\lfloor r\sqrt{n}\rfloor+k\bigr)\label{tight-7}
\end{align} 
Similarly,  
\be\sum_{m\in\bZ} P \bigl(A_{2,m}\bigr)=\sum_{k\le 0} P \bigl(X_{\lfloor nt\rfloor-\lfloor ns\rfloor}^0< \lfloor nsb\rfloor-\lfloor ntb\rfloor +\lfloor q\sqrt{n}\rfloor-\lfloor r\sqrt{n}\rfloor+k\bigr).\label{tight-8}\ee
Combining \eqref{tight-7} and \eqref{tight-8}, 
\be\label{tight-10}\begin{aligned}
&\sum_{m\in\bZ}\left[ P \bigl(A_{1,m}\bigr)+ P \bigl(A_{2,m}\bigr)\right]\\
&=\sum_{k>0} P \bigl(X_{\lfloor nt\rfloor-\lfloor ns\rfloor}^0-\lfloor nsb\rfloor+\lfloor ntb\rfloor -\lfloor q\sqrt{n}\rfloor+\lfloor r\sqrt{n}\rfloor\ge k\bigr)\\
&\qquad +\sum_{k< 0} P \bigl(X_{\lfloor nt\rfloor-\lfloor ns\rfloor}^0-\lfloor nsb\rfloor+\lfloor ntb\rfloor -\lfloor q\sqrt{n}\rfloor+\lfloor r\sqrt{n}\rfloor\le k\bigr)\\
&=\sum_{k>0} P \bigl(\bigl|X_{\lfloor nt\rfloor-\lfloor ns\rfloor}^0-\lfloor nsb\rfloor+\lfloor ntb\rfloor -\lfloor q\sqrt{n}\rfloor+\lfloor r\sqrt{n}\rfloor\bigr|\ge k\bigr)\\
&\le  E\bigl| X_{\lfloor nt\rfloor-\lfloor ns\rfloor}^0-\lfloor nsb\rfloor+\lfloor ntb\rfloor\bigr|+|r-q|\sqrt{n}+1\\
&\le  C \bigl(\sqrt{(t-s)n}+|r-q|\sqrt{n}+1\bigr).
\end{aligned}\ee
Combining \eqref{S-99} and \eqref{tight-10} gives this  bound for the third term in \eqref{tight-4}:
\be \mE\left(|\overline{S}_n(t,r)-\overline{S}_n(s,q)|^{12}\right) \le Cn^{-3}\bigl[\bigl(\sqrt{|t-s|}+|r-q|\bigr)\sqrt{n}+1\bigr]^{6}.\label{tight-11}\ee

 
Return  to   \eqref{tight-4} and apply \eqref{tight-16}, \eqref{tight-11} and \eqref{Y-101}   to conclude  that
\begin{align*}
&\mE\bigl[\,|\fluc_n(t,r)-\fluc_n(s,q)|^{12}\,\bigr]\le  Cn^{-3}\bigl[\bigl(\sqrt{|t-s|}+|r-q|\bigr)\sqrt{n}+1\bigr]^{6}\notag\\[2pt] 
&\qquad \le  C\left(|t-s|^3+|r-q|^6+n^{-3}\right) 
\le C\left(|t-s|^\kappa+|r-q|^\kappa\right). 
\end{align*}
In the last step we used $t,s,r,q\in [0,1]$,  $2<\kappa<3$,    and  
$n^{-3}\le n^{-\gamma\kappa} = \delta_n^\kappa<\abs{(t,r)-(s,q)}^\kappa$.   This completes the proof of Lemma \ref{lmm3-12}.\end{proof}


The second tightness condition is  verified as follows.   

 \begin{lemma}\label{lmm3-14}
Under the conditions of  Lemma \ref{lmm3-12}, for any fixed $1<\gamma<3/2$, for $\forall\e>0$,  
$$\lim_{n\to\infty}\mP\biggl\{ \; \sup_{\begin{subarray}{l} (t,r),(s,q)\in [0,1]^2\\ \abs{(t,r)-(s,q)}<n^{-\gamma}\end{subarray}}|\fluc_n(t,r)-\fluc_n(s,q)|>\e\biggr\} =0.$$
\end{lemma}

\begin{proof}
Define the intervals $I(k)=[(k-1)n^{-\gamma},(k+1)n^{-\gamma}]\cap [0,1]$. For  fixed $\e>0$,  first by a union bound and then by decomposition \eqref{2-7}, 
\begin{align}
&\mP\Biggl\{ \;\sup_{\begin{subarray}{l} (t,r),(s,q)\in [0,1]^2\\ \abs{(t,r)-(s,q)}<n^{-\gamma}\end{subarray}}|\fluc_n(t,r)-\fluc_n(s,q)|>\e\Biggr\} \notag\\
&\le \sum_{k_1=1}^{\lfloor n^\gamma\rfloor}\sum_{k_2=1}^{\lfloor n^\gamma\rfloor}\mP\Bigl\{\,\sup_{t\in I(k_1), r\in I(k_2)}|\mu_0\overline{H}_n(t,r)-\mu_0\overline{H}_n(k_1n^{-\gamma},k_2n^{-\gamma})|\ge\frac{\e}{6}\,\Bigr\}\label{tight-21}\\
&\qquad +\sum_{k_1=1}^{\lfloor n^\gamma\rfloor}\sum_{k_2=1}^{\lfloor n^\gamma\rfloor}\mP\Bigl\{\,\sup_{t\in I(k_1), r\in I(k_2)}|\overline{S}_n(t,r)-\overline{S}_n(k_1n^{-\gamma},k_2n^{-\gamma})|\ge\frac{\e}{6}\,\Bigr\}\label{tight-22}\\
&\qquad +\sum_{k_1=1}^{\lfloor n^\gamma\rfloor}\sum_{k_2=1}^{\lfloor n^\gamma\rfloor}\mP\Bigl\{\,\sup_{t\in I(k_1), r\in I(k_2)}|\overline{F}_n(t,r)-\overline{F}_n(k_1n^{-\gamma},k_2n^{-\gamma})|\ge\frac{\e}{6}\,\Bigr\}. \label{tight-23}
\end{align}

Line  \eqref{tight-21} vanishes for large $n$ because $\overline{H}_n(t,r)=O(n^{-1/4})$. 

For the second    and third sums  \eqref{tight-22}--\eqref{tight-23}, observe  from \eqref{2-5} and \eqref{2-6}  that $\overline{S}_n(t,r)$ and $\overline{F}_n(t,r)$ depend on their argument $(t,r)$ only through $\lfloor nt\rfloor$,  $\lfloor ntb\rfloor$ and $\lfloor r\sqrt{n}\rfloor$.  
For large enough $n$, and any $t\in I(k_1)$, $r\in I(k_2)$, 
\begin{align*}
|nt-k_1n^{1-\gamma}|&=n|t-k_1n^{-\gamma}|\le n^{1-\gamma}<1/2,\\
|ntb-k_1n^{1-\gamma}b|&=n|b|\cdot|t-k_1n^{-\gamma}|\le |b|n^{1-\gamma}<1/2,\\
|r\sqrt{n}-k_2n^{1/2-\gamma}|&=n^{1/2}|r-k_2n^{-\gamma}|\le n^{1/2-\gamma}<1/2.
\end{align*}
Thus  for $t\in I(k_1)$ and $r\in I(k_2)$, each   of $\lfloor nt\rfloor$, $\lfloor ntb\rfloor$ and $\lfloor r\sqrt{n}\rfloor$ can   have at most one jump. For example, $\lfloor nt\rfloor$ can only jump from $\lfloor k_1n^{1-\gamma}\rfloor-1$ to $\lfloor k_1n^{1-\gamma}\rfloor$ or from $\lfloor k_1n^{1-\gamma}\rfloor$ to $\lfloor k_1n^{1-\gamma}\rfloor+1$. As a result, $\overline{S}_n(t,r)$ and $\overline{F}_n(t,r)$ can take only  at most 8 different  values on $I(k_1)\times I(k_2)$. 

Suppose  $\{ \overline{S}_n(t_i,r_i)\}_{1\le i\le \ell}$ with  $\ell\le 8$ and  $(t_i,r_i)\in I(k_1)\times I(k_2)$ captures  the different  values of  $\overline{S}_n(t,r)$   on the square $I(k_1)\times I(k_2)$.   
Then,
\begin{align*}
&\mP\biggl\{ \sup_{t\in I(k_1), r\in I(k_2)}|\overline{S}_n(t,r)-\overline{S}_n(k_1n^{-\gamma},k_2n^{-\gamma})|\ge\frac{\e}{6}\biggr\}\\
&\qquad\le  \sum_{i=1}^\ell\mP\bigl\{  |\overline{S}_n(t_i,r_i)-\overline{S}_n(k_1n^{-\gamma},k_2n^{-\gamma})|\ge  {\e}/{6}\bigr\}
\\
&\qquad \le  C{n^{-3}\e^{-12} }\sum_{i=1}^\ell{\left[\left(\sqrt{|t_i-k_1n^{-\gamma}|}+|r_i-k_2n^{-\gamma}|\right)\sqrt{n}+1\right]^{6}}\\
&\qquad 
\le Cn^{-3}{\bigl[\bigl(n^{-\gamma/2}+n^{-\gamma}\bigr)\sqrt{n}+1\bigr]^{6}}\le Cn^{-3} 
\end{align*}
where the second inequality comes  from Markov's  inequality and   \eqref{tight-11}, and the third one from $\gamma>1$.
 
Therefore, for any $1<\gamma<3/2$,   as $n\to\infty$,  
\begin{align*}
&\sum_{k_1=1}^{\lfloor n^\gamma\rfloor}\sum_{k_2=1}^{\lfloor n^\gamma\rfloor}\mP\Biggl\{\,\sup_{\begin{subarray}{l} t\in I(k_1)\\ r\in I(k_2)\end{subarray}}|\overline{S}_n(t,r)-\overline{S}_n(k_1n^{-\gamma},k_2n^{-\gamma})|\ge\frac{\e}{6}\Biggr\}
\le Cn^{2\gamma-3}\longrightarrow 0.  
\end{align*}

Same reasoning with   inequality \eqref{tight-16}
shows that line \eqref{tight-23} vanishes as $n\to\infty$.   Lemma \ref{lmm3-14} has been proved.  
%
\end{proof}
The proof of Theorem \ref{thm2-6} is complete.

\appendix

\section{An ergodic lemma} 

In this section   $p(x,y)=p(0,y-x)$  
is an arbitrary   random walk kernel  on $\Z^d$ with the property that the smallest subgroup that contains the support $\{x: p(0,x)>0\}$ is $\Z^d$ itself.   

\begin{lemma}\label{lm-p-harm}
Transition probability $p(x,y)$ has no bounded harmonic functions other than constants.    
\end{lemma} 

\begin{proof}  Suppose $h$ is a bounded harmonic  function on $\Z^d$, that is,  $h(x)=\sum_y p(x,y)h(y)$ for all $x\in\Z^d$.  
Suppose $p(u,v)>0$.  Then the coupling described on p.~69 of \cite{ligg-85} works to show that $h(u)=h(v)$.  (This coupling is also spelled out in Section 1.5 of \cite{sepp-exclbook}.)  By    assumption   for any $x,y\in\Z^d$ there exists a path  $x=x_0, x_1, \dotsc, x_m=y$  in $\Z^d$ such that $p(x_i, x_{i+1})+p(x_{i+1}, x_i)>0$ for each $i=0,\dotsc,m-1$, and thereby  $h(x)=h(x_0)=\dotsm=h(x_m)=h(y)$.  
\end{proof}

Let $\cS$  be a Polish space, $\Gamma=\cS^{\bZ^d}$,  and shifts $(\theta_x\zeta)(y)=\zeta(x+y)$ for  $\zeta\in\Gamma$ and 
 $x,y\in\Z^d$.   

\begin{lemma}\label{lm-fourier}   Let $\nu$   be a probability measure on $\Gamma$  that is  invariant and  ergodic under the  shift group  and $f\in L^1(\nu)$ with  finite   mean $E^\nu[f]$.     For $x\in\bZ^d$, $t\in\bZ_+$, and $\zeta\in \Gamma$,  define 
\be\label{def-g1} g_t(x,\zeta)=\sum_{y\in\bZ}p^{t}(x,y)f(\theta_y\zeta) . 
 \ee
 Then $\forall x\in\Z^d$,  $g_t(x,\zeta)\to E^\nu[f]$ in $L^1(\nu)$ as $t\to\infty$.  
\end{lemma}

\begin{proof}  The argument is a Fourier analytic one suggested by the proof on p.~30-31 of \cite{ligg-04}.   
The characteristic function of the jump probability  is  
$$\phi_X(\alpha)=\sum_{y\in\bZ}p(0,y)e^{\iota\alpha\cdot  y}, \quad \alpha\in\bR^d.$$
For $0<r<\infty$  define   truncated functions  $f_r(\zeta)=(-r)\vee(f(\zeta)\wedge r)$ and $g_t(x,r,\zeta)=\sum_{y\in\bZ}p^{t}(x,y)f_r(\theta_y\zeta)$.  
Note that  $g_t(x,r,\zeta)$ is bounded, uniformly over $x$ and $\zeta$. 
We show that, as $t\to\infty$, $g_t(x,r,\zeta)$ converges in $L^2(\nu)$ to a constant.   An $L^1$ approximation via the truncation then implies the result. 
 
The function  $V^{(r)}(x)=E^\nu[f_r(\zeta)f_r(\theta_x\zeta)]$ is nonnegative definite   (i.e.\ $\sum_{x,y}V^{(r)}(x-y)z_x\overline{z_y}\ge 0$ for any choice of finitely many complex numbers $\{z_x\}$). By Herglotz' Theorem (Chapter \RNum{19}.6 in Feller \cite{fellII}), there exists a bounded measure $\gamma$ on $[-\pi, \pi)^d$ such that
$$V^{(r)}(x)=\int e^{-\iota x\cdot \alpha}\gamma(d\alpha), \quad  x\in\bZ^d.$$ 

Let   $X_t$ and $\tilde{X}_t$  be   i.i.d.\ copies of the random walk with transition   $p$.  Compute:
\begin{align*}
&\int g_t(x,r,\zeta)g_s(x,r,\zeta)\nu(d\zeta)= \int  E ^{x}[f_r(\theta_{X_t}\zeta)] E ^{x}[f_r(\theta_{\tilde X_s}\zeta)]\nu(d\zeta)\\
&=  E ^{(x,x)}\int f_r(\theta_{X_t}\zeta) f_r(\theta_{\tilde X_s}\zeta) \nu(d\zeta)= E ^{(x,x)} \bigl[V^{(r)}(\tilde{X}_s-X_t)\bigr]\\
&= E ^{(x,x)}\int e^{-\iota\alpha\cdot (\tilde{X}_s-X_t)}\gamma(d\alpha)=\int \overline{ E ^x e^{\iota\alpha\cdot\tilde{X}_s}}\cdot E ^x e^{\iota\alpha\cdot X_t}\gamma(d\alpha)\\
&=\int [\overline{\phi_X(\alpha)}]^s[\phi_X(\alpha)]^t\gamma(d\alpha)=\int [\phi_X(\alpha)]^s[\overline{\phi_X(\alpha)}]^t\gamma(d\alpha).
\end{align*}
The last equality came by switching $s$ and $t$ around in the calculation.  
From this, 
\begin{align}
&\int \bigl[g_t(x,r,\zeta)-g_s(x,r,\zeta)\bigr]^2 \nu(d\zeta) \nn\\
&= \int \bigl[g_t(x,r,\zeta)^2-2g_t(x,r,\zeta)g_s(x,r,\zeta)+g_s(x,r,\zeta)^2\bigr] \nu(d\zeta) \notag\\
&=\int \bigl[\left|\phi_X(\alpha)\right|^{2t}-2\overline{\phi_X(\alpha)^s}\phi_X(\alpha)^t+\left|\phi_X(\alpha)\right|^{2s}\bigr]\gamma(d\alpha)=\int \left |\phi_X(\alpha)^t-\phi_X(\alpha)^s\right|^2 \gamma(d\alpha)\notag\\
&=\int_{\alpha\neq 0} \left |\phi_X(\alpha)^t-\phi_X(\alpha)^s\right|^2 \gamma(d\alpha).\label{u-6}
\end{align}
Since   $\Z^d$ is the smallest subgroup that contains the support of  $p$, $|\phi_X(\alpha)|<1$  $\forall\alpha\in [-\pi,\pi)^d\setminus \{0\}$ (T7.1 in \cite{spitzer}). Thus the integrand in \eqref{u-6} is bounded and  converges to zero as $s, t\to\infty$.  Hence  $\{g_t(x,r,\zeta)\}_{t\in\Z_+}$ is  Cauchy   in $L^2(\nu)$ and  $\exists$  a bounded     $L^2(\nu)$ limit
$\bar g(x,r, \zeta)=\lim_{t\to\infty}g_t(x, r, \zeta).$

Letting $s\to\infty$ in 
$
g_{s+t}(x,r,\zeta)=\sum_y p^t(x,y) g_s(y,r, \zeta)
$
implies 
\[ \bar g(x,r,\zeta)=\sum_y p^t(x,y) \bar g(y,r,\zeta). \]  
Hence  for almost every $\zeta$,  $\bar g(\cdot\,,r,\zeta)$ is a bounded harmonic function for $p(x,y)$ and thereby a constant in $x$.  Combining this with a shift gives
$\bar g(0,r,\zeta)=\bar g(x,r,\zeta)=\bar g(0, r,  \theta_x\zeta)$ $\forall x\in\Z^d$. 
By ergodicity,  $\bar g(0,r,\zeta)$ equals $\nu$-almost surely a constant, and hence 
$\bar g(x,r,\zeta)$ equals $\nu$-almost surely  the same constant for all $x\in\Z^d$.   
 
Now we transfer these properties to $g_t(x,\zeta)$  through $L^1$ approximation.  
  By shift-invariance 
\[  \norm{g_t(x,\zeta)-g_t(x,r,\zeta)}_1 \le \sum_{y}p^{t}(x,y)\norm {f\circ\theta_y-f_r\circ\theta_y}_1 = \norm {f-f_r}_1 \; \underset{r\to\infty}\longrightarrow \; 0.
\]
From this we deduce  that $\{g_t(x,\cdot)\}$ is a Cauchy sequence in $L^1(\nu)$.  Hence 
we have an   $L^1(\nu)$ limit   $g(x,\zeta)=\lim_{t\to\infty}g_t(x, \zeta)$ and we can take $t\to\infty$ in the bound above to get 
\[  \norm{g(x,\zeta)-\bar g(x,r,\zeta)}_1 \le  \norm {f-f_r}_1. 
\]
Letting $r\to\infty$ takes  the right-hand side to zero, and we  
conclude that $g(x,\zeta)$ is $\nu$-almost surely a constant independent of $x$. 
This constant must equal $E^\nu[f]$  because by the $L^1(\nu)$ convergence 
$E^\nu[g(x,\cdot)]=\lim_{t\to\infty}  E^\nu[g_t(x,\cdot)]=E^\nu[f]$.  
\end{proof}

\bibliographystyle{plain}
\bibliography{harness_processes_in_one_spatial_dimension,growthrefs}
\end{document}